\documentclass[a4paper,12pt,reqno,twoside]{amsart}

\usepackage[utf8]{inputenc} 		% UTF8 encoding
\usepackage[T1]{fontenc} 

\usepackage{dirtytalk}

\usepackage{amssymb,amsmath}
\usepackage[foot]{amsaddr}
\usepackage{graphicx}
\usepackage[colorlinks,citecolor=blue,urlcolor=blue]{hyperref}
\usepackage{cite}
\usepackage[hmargin=2.5cm,vmargin=2.5cm]{geometry}
\usepackage{mathrsfs} % Script fonts
\usepackage{ellipsis}

\usepackage{pgfplots}
\usepackage{mathtools}
\usetikzlibrary{math,calc}

\usepackage{tikz}
\usetikzlibrary{automata,arrows,positioning,calc}
\pgfplotsset{compat=1.18}

\newcommand{\Lip}{{\mathrm{Lip}}}
\newcommand{\LL}{{\mathrm{LL}}}

\newcommand{\be}{\begin{eqnarray}}
\newcommand{\ee}{\end{eqnarray}}

\newcommand{\beq}{\begin{equation}}
\newcommand{\eeq}{\end{equation}}
\newcommand{\beqn}{\begin{equation*}}
\newcommand{\eeqn}{\end{equation*}}

\newcommand{\round}[1]{\lfloor#1\rfloor}

\newtheorem{thm}{Theorem}[section]
\newtheorem{prop}[thm]{Proposition}
\newtheorem{cor}[thm]{Corollary}
\newtheorem{lem}[thm]{Lemma}
\newtheorem{defn}[thm]{Definition}

\theoremstyle{remark}
\newtheorem{remark}[thm]{Remark}
\newtheorem{example}[thm]{Example}

\newcommand\sL{{\mathscr{L}}}
\newcommand\sM{{\mathscr{M}}}

\newcommand\cA{{\mathcal A}}
\newcommand\cB{{\mathcal B}}
\newcommand\cC{{\mathcal C}}

\newcommand\cE{{\mathcal E}}
\newcommand\cF{{\mathcal F}}

\newcommand\cH{{\mathcal H}}
\newcommand\cI{{\mathcal I}}

\newcommand\cL{{\mathcal L}}

\newcommand\cP{{\mathcal P}}
\newcommand\cQ{{\mathcal Q}}

\newcommand\bP{{\mathbb P}}

\newcommand\bR{{\mathbb R}}

\newcommand\bT{{\mathbb T}}

\newcommand\bZ{{\mathbb Z}}

\newcommand\fD{{\mathfrak D}}

\newcommand\fF{{\mathfrak F}}
\newcommand\fG{{\mathfrak G}}

\newcommand{\ve}{\varepsilon}

\def\bfE{\mathbf{E}}

\def\bfI{\mathbf{I}}

\begin{document}
	
\title[Brownian approximation of dynamical systems]{Error bounds in a smooth metric for Brownian approximation of dynamical systems via Stein's method}

\author[Juho Lepp\"anen]{Juho Lepp\"anen$^\dagger$}
\address{$^\dagger$Department of Mathematics, Tokai University, Kanagawa, 259-1292, Japan}
\email{leppanen.juho.heikki.g@tokai.ac.jp}

\author[Yuto Nakajima]{Yuto Nakajima$^\S$}
\email{yunakaji@mail.doshisha.ac.jp}
\address{$\S$Faculty of Science and Engineering, Doshisha University, Kyoto, 610-0394, JAPAN}
\author[Yushi Nakano]{Yushi Nakano$^\ddagger$}
\address{$^\ddagger$Faculty of Science, Hokkaido University, Hokkaido, 060-0810, Japan}
\email{yushi.nakano@math.sci.hokudai.ac.jp}

 % Suggesting these. 
\thanks{J. Lepp\"anen and Y. Nakajima were supported by JSPS via the project LEADER, and 
	Y. Nakajima by JSPS KAKENHI Grant Number 25K17282.
	Y. Nakano was supported by JSPS KAKENHI Grant Number 23K03188.
	J. Lepp\"anen would like to thank Mikko Stenlund for helpful correspondence.
	We thank the anonymous referees for their pertinent comments, which helped us correct errors and improve the paper.
	}

\maketitle

\begin{abstract}	
We adapt Stein's method of diffusion approximations, developed by Barbour, 
to the study of chaotic dynamical systems. We establish an error bound in the functional central limit theorem with respect to an integral probability metric of smooth test functions under a functional correlation decay bound. 
For systems with a sufficiently fast polynomial rate of correlation decay, the error bound is of order $O(N^{-1/2})$,  under an additional condition on the linear growth of variance.
Applications include a family of interval maps with neutral
fixed points and unbounded derivatives, and two-dimensional dispersing Sinai billiards.
\end{abstract}

\section{Introduction}

The functional central limit theorem (FCLT), also known as Donsker's invariance principle 
\cite{D51}, asserts that for a sequence $(X_n)_{n \ge 0}$ 
of i.i.d. random variables with zero mean and unit variance, the random process 
$$
W_N(t) = N^{-1/2} \sum_{n = 0}^{  \round{  N t  } - 1  } X_n, \quad t \in [0,1],
$$
converges in distribution to a standard Brownian motion 
$(Z(t))_{t \in [0,1]}$ 
with respect to the Skorokhod topology. Barbour \cite{B90} observed that Stein's method \cite{S72},  initially developed for error estimation in the central limit theorem, can be adapted 
to obtain rates of convergence in the FCLT with respect to integral probability metrics of sufficiently 
smooth test functions.

Consider an integral probability metric of the form 
\begin{align}\label{eq:integral_metric}
d_{\fG}(\mu_N, \nu) := \sup_{g \in \fG} |  \mu_N( g ) -  \nu(g)  |,
\end{align}
where $\fG$ is some class of real-valued test functions, $(\mu_N)_{N \ge 0}$ is a sequence 
of probability distributions, and $\nu$ is a known target distribution used to approximate 
$\mu_N$. Here, $\nu(g)$ denotes the expectation of $g$ with respect to $\nu$.
The core idea behind Stein's method can be summarized as follows:
to estimate \eqref{eq:integral_metric}, a linear operator $\cA$, called a Stein operator, together 
with a class of functions $\fF(\cA)$ is constructed such that 
(a) the identity $\nu( \cA f  ) = 0$ holds for  each $f \in \fF(\cA)$; and (b) for each 
$g \in \fG$, there exists a solution $f \in \fF(\cA)$ to the Stein equation
\begin{align}\label{eq:se-intro}
\cA f (w) = g(w) - \nu( g ).
\end{align}
Taking expectations reformulates the original problem of estimating \eqref{eq:integral_metric} into one of estimating 
$\sup_{f \in \fF(\cA)} | \mu_N ( \cA f ) |$, for which
various techniques have been developed in the extensive literature on 
Stein's method.
For continuous target distributions, these techniques commonly involve Taylor expansions, coupling methods, and Malliavin calculus.
A notable feature of Stein's method, particularly from the viewpoint of applications,
is its flexibility in handling various dependence structures
including local dependence \cite[Chapter 9]{CGS11} and weak dependence \cite[Chapter III]{P00}.
For further background on Stein's method and its applications in different probabilistic frameworks, we refer the reader to the monographs \cite{CGS11, NP12} and the surveys \cite{R11,LRS17}.

In the special case where the target distribution $\nu$ is the 
Wiener measure, Barbour considered a class $\sM$ (defined below in Section \ref{sec:spaces})
 of test functions
$g : D \to \bR$ on the space $D := D([0,1], \bR)$ of c\`{a}dl\`{a}g functions $w : [0,1] \to \bR$
equipped with the sup-norm, where $g$ is
twice Fr\'{e}chet 
differentiable with a Lipschitz continuous second derivative and satisfies appropriate growth constraints.
To construct a Stein operator $\cA$, Barbour used Markov generator theory, defining $\cA$ as the infinitesimal generator of a Markov process that solves an SDE with stationary distribution $\nu$. 
Using the transition semigroup associated with $\cA$, he identified the 
solution to \eqref{eq:se-intro} and analyzed its properties. 
Employing
Taylor expansions in combination with the leave-one-out approach, he derived an error bound of the form
\begin{align}\label{eq:error_intro}
| \bfE[ g( W_N ) ]  - \bfE[ g(Z) ] | \le C N^{-1/2} \Vert g \Vert ( \sqrt{  \log (N) } + \bfE[ |X_1|^3 ] ),
\end{align}
where $C > 0$ is an absolute constant and $\Vert g \Vert$ is a suitable norm of $g$. Recently, Kasprzak \cite{K20} extended Barbour's results \cite{B90} to the approximation of vector-valued processes by multivariate correlated Brownian motion, where the covariance matrices can be non-identity and the approximated process is allowed to have certain dependence structure. Barbour, Ross, and Zheng \cite{BRZ24} derived Gaussian smoothing inequalities that, in the case of error bounds such as \eqref{eq:error_intro} obtained via Barbour's method, allow the class of test functions to be extended at the cost of reduced precision in the estimates.

Our aim in this study is to adapt Barbour's technique to extend error bounds 
such as 
\eqref{eq:error_intro} to cases where the process
$(X_n)$ is generated by a deterministic 
dynamical system with good mixing properties.
We restrict our considerations to self-normed c\`{a}dl\`{a}g
paths $W_N(t)$ (see \eqref{eq:def_W}) and approximations by the univariate standard 
Brownian motion, with generalizations (such as those considered in \cite{K20})
left for future exploration\footnote{This choice allows us to use the definition of the Stein equation from \cite{B90} without modification.}.
Our study is, in spirit, a continuation of previous studies \cite{HLS20, LS20} that dealt with 
Gaussian approximation of dynamical systems using Stein's method.
An observation from Barbour's work \cite{B90} is that, 
a rate of convergence in Donsker's theorem follows with little added effort to the Gaussian 
case, once Stein's method for Brownian approximations has been properly set up. 
Based on results from the present study and those in \cite{HLS20, LS20}, we find that this observation continues to hold for a broad class of chaotic dynamical systems.

The problem of estimating rates of convergence in FCLTs for dynamical systems has been addressed in \cite{P24, AM19, LW24-1, DMR24, F23, F25, H23}. 
Using an explicit construction of an approximating Gaussian sequence together with a Fuk--Nagaev type deviation inequality, Dedecker, Merlev\`ede, and Rio \cite{DMR24} established rates of order $O(N^{-1/4}(\log N)^{1/4})$ for the Wasserstein-2 metric in the case of Young towers with return times to the base having a finite fourth moment. 
A series of works has been devoted to the adaptation of martingale approximation techniques, starting with Antoniou and Melbourne \cite{AM19}, who derived convergence rates with respect to the Prokhorov metric. 
For systems modeled by Young towers with exponential tails of return times to the base, such as planar periodic dispersing billiards and unimodal maps satisfying the Collet--Eckmann condition, their method yields rates of order $O(N^{-1/4+\delta})$ for arbitrarily small $\delta > 0$. 
For Young towers with polynomial tails, including interval maps with neutral fixed points as in the Pomeau--Manneville scenario \cite{PM80}, the rates depend on the \say{degree of nonuniformity.}
In a setting similar to \cite{AM19}, Liu and Wang \cite{LW24-1} obtained rates for Wasserstein-$p$ metrics, and Paviato \cite{P24} established rates in multidimensional FCLTs with respect to Prokhorov and Wasserstein-$1$ metrics.
In \cite{LW24-2}, Liu and Wang derived convergence rates for Wasserstein-$p$ metrics in the FCLT for uniformly expanding nonautonomous dynamical systems described by concatenations $T_n \circ \cdots \circ T_1$ of varying maps $T_i : X \to X$, such as those in the setting of Conze and Raugi \cite{CR07}. 
Their rates are $O(\sigma_N^{-1/2+\delta})$ for arbitrarily small $\delta > 0$, where $\sigma_N^2$ denotes the variance of the partial sum process. 
Under linear growth of $\sigma_N^2$ this corresponds to $O(N^{-1/4+\delta})$.

The primary contribution of our work is the method, which, under certain weak dependence criteria (see 
Definition \ref{defn:fcb})
called functional correlation bounds,
yields an error bound in a self-normed FCLT with respect to a metric as in \eqref{eq:integral_metric}.
The result applies to test functions $g : D \to \bR$ that belong to Barbour's class $\sM$ and satisfy 
an additional smoothness condition (see \eqref{eq:smooth}). The error decays at the rate $O(N^{-1/2})$, provided 
that the decay rate of functional correlations associated to separately Lipschitz functions
has a finite first moment and that
the variance of the partial sum $\sum_{n=0}^{N-1} X_n$ grows linearly as $N \to \infty$.
We consider two types of applications:
(1) stationary processes of the form 
$X_n = f \circ T^n$ with $T^n = T^{n-1} \circ T$, where $T : M \to M$ is a measure-preserving transformation 
of a probability space $(M, \cF, \mu)$ and $f : M \to \bR$ is a suitably regular observable; 
and (2) nonstationary processes of 
the form $X_n = f_n \circ T_n \cdots \circ T_1$, 
where $(T_n)$ is a 
deterministic/random
sequence of 
transformations and $(f_n)$ is a sequence of regular observables. 

Prior to our work, Fleming \cite{F23} addressed the problem of estimating the rate of convergence in FCLTs for ergodic measure-preserving 
transformations satisfying weak dependence criteria also called
functional correlation bounds, which are closely related yet different from those considered here (compare \cite[Definition 3.2.2]{F23} with \eqref{eq:fcb}).
Using an approach based on Bernstein's blocking technique, he derived convergence rates in the multivariate FCLT with respect to the Wasserstein-1 metric of Lipschitz continuous test functions, at best of order $O(N^{-1/4 + \delta})$, 
where the rate of convergence depends on the decay rate of the functional correlation bounds. 
In the recent work \cite{F25}, similar techniques are used to derive Wasserstein-$p$ rates in multivariate FCLTs. In the case of Young towers with superpolynomial tails, rates of order $O(n^{-\kappa})$ for all $\kappa < 1/4$ and $p < \infty$ are established.

Previously, functional correlation bounds were derived in 
\cite{LS20, L17, LS17} for piecewise uniformly expanding interval maps, Pomeau--Manneville-type
interval maps, and a class of dispersing Sinai billiards.
By combining these bounds with our abstract theorem
(Theorem \ref{thm:main}), we establish order $O(N^{-1/2})$ error bounds in FCLTs 
for these systems. To our knowledge, all of these results are new.
Additionally, we analyze an interval map 
$T : [-1,1] \to [-1,1] $ introduced by Pikovsky \cite{P91}, 
which features two neutral fixed points and an unbounded derivative.
The neutrality of the fixed points and the degree of the singularity 
are governed by a single parameter  $\gamma \in (1, \infty)$.
By \cite{CHMV10, CLM23}, this map 
admits a Young tower with a first return map whose tails 
decay at the rate $O(n^{- \gamma / ( \gamma - 1)})$. An immediate consequence 
(see e.g. \cite{MN05, AM19}) is that,
for $\gamma < 2$, the process 
$X_n = f \circ T^n$ where $f$ is Lipschitz continuous
satisfies 
the FCLT.
We prove that $(X_n)$
satisfies a functional correlation bound with a polynomial rate that matches the decay rate of usual auto-correlations from 
 \cite{CHMV10}. This leads to new error bounds in the FCLT for parameters $\gamma < 5/3$, 
 which decay at the order $O(N^{-1/2})$ when $\gamma < 3/2$.
  
 \begin{remark}
After circulating a preliminary version of this manuscript, we learned from Yeor Hafouta about the related interesting work in \cite[Chapter 1]{HK18}, which applies Barbour's method \cite{B90} in the context of weakly dependent processes. 
In particular, \cite[Theorem 1.6.2]{HK18} provides error bounds in the FCLT for nonconventional sums of 
the form $\sum_{n=1}^{  \round{Nt} } F( \xi_n, \xi_{2n}, \ldots, \xi_{\ell n})$ where 
$(\xi_n)$ satisfies certain stationarity and mixing conditions and $F$ is a sufficiently regular function. Dynamical applications include topologically mixing subshifts of finite type and systems 
with corresponding symbolic representations, such as smooth Axiom A diffeomorphisms.
While the metric used in \cite{HK18} does not impose the smoothness condition \eqref{eq:smooth}, the established error estimates decay at a rate slower than $O(N^{-1/2})$, depending on the rate 
of mixing.
 \end{remark}
 
\subsection*{Structure of the Paper}
In Section \ref{sec:main_thm}, we present our main result (Theorem \ref{thm:main}) concerning the rate of convergence in the FCLT with respect to an integral probability metric of smooth test functions for sequences of uniformly bounded real-valued random variables. 
The proof, given in Section \ref{sec:main_proof} and Appendix \ref{sec:proof_estimates_Ei}, relies on Barbour's method \cite{B90} which 
is reviewed in Section \ref{sec:stein}.
Our hypothesis, similar to \cite{LS20, HLS20}, is a functional correlation bound with a sufficiently fast polynomial rate of decay. Examples of dynamical systems that satisfy such bounds are discussed in Section \ref{sec:appli}. For the intermittent map
of Pikovsky \cite{P91}, we prove a polynomial functional correlation bound in Section \ref{sec:proof_pikovsky}.

\section{Abstract theorem}\label{sec:main_thm}

\subsection{Spaces of test functions $\mathscr{L}$, $\mathscr{M}$, and  $\mathscr{M}_0$}\label{sec:spaces} Denote by $D$ the space of all c\`{a}dl\`{a}g functions $w : [0,1] \to \bR$ equipped with the sup-norm $\Vert w \Vert_\infty = \sup_{ t \in [0,1] } | w(t)|$.
Given a function $f : D \to \bR$, by $f^{(k)}$ we mean the $k$th Fr\'{e}chet derivative of $f$, which is 
a map 
$f^{(k)} : D \to \cL( D^k , \bR )$ from $D$ to the space
$\cL( D^k, \bR )$ of all continuous 
multilinear maps from $D^k$ to $\bR$. The $k$-linear norm of $A \in \cL( D^k, \bR )$
is defined by
$$
\Vert A \Vert = \sup_{ \Vert w_i \Vert_\infty \le 1 \: \forall i = 1,\ldots, k } | A[w_1,\ldots, w_k] |,
$$
where 
$A[w_1,\ldots, w_k]$ denotes $A$ applied to the arguments $w_1,\ldots, w_k \in D$.

Following \cite{B90}, 
let $\sL$ be the Banach space of all continuous functions $g : D \to \bR$ for which the norm  
$$
\Vert g \Vert_{\sL} := \sup_{w \in D} \frac{ | g(w) | }{  1 + \Vert w \Vert_\infty^3 }
$$
is finite. Let $\sM \subset \sL$ be the subcollection of all twice Fr\'{e}chet differentiable functions $g \in \sL$ that satisfy
\begin{align}\label{eq:d2-lip}
\sup_{w,h \in D, \, h \neq 0} \frac{\Vert g''(w + h) -  g''(w)   \Vert}{  \Vert h \Vert_\infty } < \infty.
\end{align}
A norm on $\sM$ can be defined as follows:

\begin{prop}[See equation (2.7) in \cite{B90}]\label{prop:M_norm} For every $g \in \sM$, define 
	\begin{align*}
		\Vert g \Vert_\sM &= \sup_{w \in D} \frac{ | g(w) | }{ 1 + \Vert w \Vert_\infty^3 } +  
		\sup_{w \in D} \frac{  \Vert  g' (w) \Vert   }{ 1 + \Vert w \Vert_\infty^2 }  + 
		\sup_{w \in D} \frac{  \Vert  g'' (w) \Vert^2   }{ 1 + \Vert w \Vert_\infty } \\
		&+ 
		\sup_{w,h \in D}  \frac{ \Vert g''(w + h) - g''(w)  \Vert  }{ \Vert h \Vert_\infty }
	\end{align*}
	Then, for all $g \in \sM$, we have $\Vert g \Vert_\sM < \infty$.
\end{prop}

\begin{remark}
For a test function $g \in \sM$, the solution $\phi(g)$ to the 
Stein equation \eqref{eq:se}, described in Section \ref{sec:stein}, 
satisfies $\phi(g) \in \sM$; see Lemma \ref{lem:stein_sol}. 
\end{remark}

To facilitate the adaptation of Stein's method to Brownian approximations of dynamical systems, an additional regularity assumption is imposed on $g \in \sM$.
Let $\sM_0 = \sM_{0}(C_0) \subset \sM$ consist of all $g \in \sM$ that satisfy the following smoothness condition introduced in \cite{B90}:
\begin{align}\label{eq:smooth}
	\sup_{w \in D}	| g''(w)[J_r, J_s - J_t] |   
	\le C_0 \Vert g \Vert_{\sM } |t-s|^{1/2} \quad \forall r,s,t \in [0,1],
\end{align}
where 
\begin{align}\label{eq:Jt}
	J_{ \alpha }(t) = \begin{cases}
		1, & \text{if $t \ge \alpha$}, \\ 
		0, & \text{if $t < \alpha$}.
	\end{cases}
\end{align}

\begin{remark}
	Condition \eqref{eq:smooth} is used in two places:
	Lemma \ref{lem:generator} and Proposition \ref{prop:estim_ii}. 
	For Lemma \ref{lem:generator}, which concerns the definition of the Stein operator $\cA$, 
	the following weaker condition suffices:
		\begin{align*}
		\lim_{n \to \infty} 
		\int_0^1 g''(w)[ J_{  \round{nt} /n   }^{(2)} ] \, dt = \int_0^1 g''(w) [  J_t^{(2)} ] \, dt \quad \forall w \in D.
	\end{align*}
	The formulation of \eqref{eq:smooth} is specifically tailored for Proposition \ref{prop:estim_ii}, which is instrumental for
	obtaining the error bound in Theorem \ref{thm:main}. Various relaxations of \eqref{eq:smooth}, such as weakening the modulus of 
	continuity $|t-s|^{1/2}$ or allowing $| g''(w)[J_r, J_s - J_t] | $ to grow with $\Vert w \Vert_\infty$ could 
	be considered. However, these would affect the error bounds in Theorem \ref{thm:main} and the rates of convergence in the FCLTs discussed in Section \ref{sec:appli}.
\end{remark}

\begin{example} The collection $\sM$
contains functions of the form
$$
g(w) = \int_0^1 \kappa (t, w(t)) \, d m(t)
$$
where $\kappa : [0,1] \times \bR \to \bR$ is a measurable function such that $\kappa (t, \cdot) \in C^2$ for all $t \in [0,1]$, 
\begin{align*}
&\sup_{t \in [0,1]} | \kappa (t,0)| + \sup_{t \in [0,1]} | \partial_2  \kappa (t,0)|  + \sup_{t \in [0,1]} | \partial_2^2  \kappa (t,0)| \le C_{\kappa}, \\
&\sup_{t \in [0,1]}  |  \partial_2^2 \kappa (t, x )  - \partial_2^2 \kappa (t,y)  | \le C_{\kappa} |x - y| \quad \forall x,y \in \bR^2,
\end{align*}
and $m$ is a probability measure on $[0,1]$. If we further require that 
$$
m([t,s]) \le C_m |t - s|^{1/2} \quad \text{and} \quad  \sup_{t \in [0,1], x \in \bR} | \partial_2^2  \kappa (t,x)| \le C_{\kappa}',
$$ 
then $\sM_0(C_0)$ contains $g$ for $C_0 = 2 C_m C_\kappa'$. Indeed, 
the second Fr\'{e}chet derivative of $g$ is 
given by
$$
g''(w)[h_1, h_2] = \int_0^1  \partial_2^2 \kappa (t, w(t)) h_1(t) h_2(t)  \, d m(t).
$$
Therefore,
\begin{align*}
	| g''(w)[J_r, J_s - J_t] | &\le C_{\kappa}' m([s,t])  \le C_{\kappa}' C_m |t - s |^{1/2} \le  \Vert g \Vert_{\sM}  2 C_{\kappa}' C_m |t - s |^{1/2}.
\end{align*}
\end{example}

\begin{example} For $t,s \in [0,1]$, the following function, which can be used to identify the covariance 
structure of the limiting process in the FCLT, belongs to $\sM$:
$$
g(w) = w(t)w(s).
$$
However, $g$ does not belong to $\sM_0(C_0)$ for any $C_0$, as it does not satisfy \eqref{eq:smooth}.
On the other hand, $\sM_0(C_0)$ for suitably 
large $C_0 = C_0(\ve)$ does contain 
the following smooth approximation of $g$:
$$
g_\ve(w) = \int_0^1 \int_0^1 K_\ve (  t - u ) K_\ve (  s - v )  w (u) w(v) \, du dv, 
$$
where $K(x) = ( 2 \pi )^{-1 /2 } e^{-x^2/2}$ is a Gaussian kernel and $K_\ve(x) = \ve^{-1} K(x/\ve)$.
The second Fr\'{e}chet derivative of $g_\ve$ is 
$$
g_\ve''(w)[h_1, h_2] =  \int_0^1 \int_0^1 K_\ve (  t - u )  K_\ve (  s - v )  h_1 (u) h_2(v) \, du dv.
$$
An elementary computation yields
\begin{align*}
		| g_\ve''(w)[J_r, J_s - J_t] | \le C_\ve \Vert g_\ve \Vert_{\sM} |s -t|
\end{align*}
for some constant $C_\ve > 0$ depending on $\ve$.
\end{example}

\begin{example}
	Let $\Theta \sim N(0,1)$ and $U \sim \mathrm{Uniform}(0,1)$ be random variables independent of $Z$, 
	where $Z$ is a standard Brownian motion.
	For $\varepsilon, \delta > 0$ and $h : D \to \bR$, define 
	$$
	h_{\varepsilon, \delta}(w) = \bfE [ h ( w_\varepsilon + \delta Z + \delta \Theta ) ],
	$$
	where $w_\varepsilon$ is given by
	\begin{align}\label{eq:w_eps}
	w_\varepsilon(s) = \bfE [ w ( s + \varepsilon U ) ],
	\end{align}
	with the convention that $w(t) = w(0)$ if $t < 0$ and $w(t) = w(1)$ if $t > 1$.  
	It was proved in \cite[Lemma 1.11]{BRZ24} that for any Skorokhod-measurable function 
	$h$ satisfying $\sup_{w \in D} |h(w)| \le 1$, 
	$$
	h_{\ve, \delta} \in \mathscr{M}_0(C_0) 
	\quad \text{with} \quad 
	C_0 = O \bigl( \ve^{-2} \delta^{-2} \bigr).
	$$
	The functions $h_{\varepsilon, \delta}$ were used in \cite{BRZ24} 
	to derive Gaussian smoothing inequalities, which 
	can be applied to 
	extend FCLT error bounds from $g \in \mathscr{M}_0$ to broader classes of test functions, 
	including bounded Lipschitz functions \cite[equation (1.15)]{BRZ24}
	and indicators of Skorokhod-measurable sets \cite[equation (1.14)]{BRZ24}, 
	albeit at the cost of reduced precision in the estimates.
\end{example}

\subsection{Functional correlation bound} Let $X_0, \ldots, X_{N-1}$, $N \ge 1$, be real-valued random variables defined on a probability space $(M, \cF, \mu)$ such that 
\begin{align}\label{eq:assume_rv}
\Vert X_n \Vert_\infty \le L \quad \text{and} \quad	\mu(X_n) = 0, \quad \text{for all $0 \le n < N$,}
\end{align}
for some constant $L > 0$, where $\mu(X_n)$ denotes the expectation of $X_n$ with respect to $\mu$\footnote{The assumption $\mu(X_n) = 0$ will not be restrictive because $\mu(\tilde X_n) = 0$ holds for $\tilde X_n:=X_n-\mu(X_n)$ and $X_n$ is integrable by the other assumption $\Vert X_n\Vert_\infty\le L$.}. 
For a finite non-empty subset $$I = \{i_1, \ldots, i_n\} \subset \bZ_+ \cap [0, N-1]$$ 
of indices $i_1 < \ldots < i_n$, we denote 
by $\nu_I$ the joint distribution of the 
subsequence $(X_i)_{i \in I}$. That is, 
$\nu_I$ is a probability measure on 
$[-L,L]^{n}$ characterized 
by the identity
\begin{align*}
	\int_{[-L,L]^n} h \, d \nu_I 
	= \int _{M} h(  X_{i_1} , \ldots, 
	X_{i_n} ) \, d \mu
\end{align*}
for bounded measurable $h : [-L,L]^n \to \bR$. In what 
follows, we consider unions 
$ I = \cup_{1 \le k \le K} I_k$ of 
index sets $I_k = \{i_{p_{k-1} + 1}, \ldots, 
i_{p_k} \}$ with $i_{p_{k-1} + 1} < \ldots < i_{p_k}$.
We will always assume that the 
sets are disjoint and ordered, in the sense that 
the gap between $I_k$ and $I_{k+1}$ satisfies 
$$
g_{k} = i_{p_k + 1} - i_{p_k} > 0 \quad \forall 
k = 1, \ldots, K-1.
$$
For brevity, we shall henceforth write $I_1 < \cdots < I_K$ 
to express these conventions.

\begin{defn} 
	Let $\vartheta \in (0,1]$. Given a function $F : [-L,L]^k \to \bR$, where $k \geq 1$, we define 
	\[
	[F]_{\vartheta} = \max_{1 \leq i \leq k} \sup_{x \in [-L,L]^k} 
	\sup_{a \neq a'} 
	\frac{|F(x(a/i)) - F(x(a'/i))|}{|a - a'|^{\vartheta}}.
	\]
	Here, $x(a/i) \in [-L, L]^k$ denotes the vector obtained from $x \in [-L,L]^k$ by replacing the $i$th component $x_i$ with $a \in [-L, L]$. 
	We say that $F$ is separately H\"{o}lder continuous with exponent $\vartheta$ if $[F]_{\vartheta} < \infty$, and we define $\|F\|_{\vartheta} = \|F\|_\infty + [F]_{\vartheta}$.
	Moreover, if $[F]_{\vartheta} < \infty$ holds with $\vartheta = 1$, we say that $F$ is separately Lipschitz continuous and write $\|F\|_{\Lip} = \|F\|_1$.
\end{defn}

\begin{defn}\label{defn:fcb}
	We say that $(X_n)_{0 \le n < N}$ 
	satisfies a functional 
	correlation bound with rate function 
	$R : \{1,2,\ldots\} \to \bR_+$ and constant $C_*$, if
	the following holds. Whenever 
	$I \subset \bZ_+ \cap [0, N -1]$ and  
	$I_1 < I_2$ are such that  
	$I = I_2 \cup I_2$, and  
	$F : [-L, L]^{|I|} \to \bR$ is 
	a separately Lipschitz continuous function,
	\begin{align}\label{eq:fcb}
		\biggl| \int F \, d \nu_{I} 
		- \int F d ( \nu_{I_1} \otimes  \nu_{I_2} ) \biggr|
		\le C_* \Vert F \Vert_{ \textnormal{Lip} } R(g_1). \tag{FCB}
	\end{align}
\end{defn}

By induction, \eqref{eq:fcb} readily extends  to the 
case of $K$ index sets ($K-1$ gaps) as follows. 

\begin{prop}\label{prop:fcb1-k}
	If $(X_n)_{0\le n< N}$ satisfies the functional correlation bound
	with rate function $R : \{1,2,\ldots\} \to \bR_+$ and constant $C_*$, then the following holds for all $K \ge2$: Whenever 
	$I \subset \bZ_+ \cap [0, N -1]$ and  
	$I_1 < \cdots < I_K$ are such that  
	$I = \cup_{k=1}^K I_k $, and  
	$F : [-L, L]^{|I|} \to \bR$ is 
	a separately Lipschitz continuous function,
	\begin{align}
		\label{eq:fcb2}
		\biggl| \int F \, d \nu_{I} 
		- \int F d ( \nu_{I_1} \otimes 
		\cdots \otimes \nu_{I_{ K } } ) \biggr|
		\le C_* \Vert F \Vert_{ \textnormal{Lip} }  \sum_{k=1}^{K-1} R(g_k).
		\tag{FCB'}
	\end{align}
\end{prop}

\begin{proof}
	This is done by induction with respect to $K\ge 2$. 
	The base case \(K = 2\) is given by the assumption that \eqref{eq:fcb} holds.
	We now suppose that \eqref{eq:fcb2} holds for $K-1$ where $K > 2$. 
	Since \eqref{eq:fcb} holds for $K=2,$ we have
	\begin{align}\label{eq:fcb_base_case}
		\biggl|\int F \ d\nu_I-\int F \ d(\nu_{\cup_{i=1}^{K-1}I_i}\otimes \nu_{I_K}) \biggr|
		\le C_{\ast}\Vert F \Vert_{ \textnormal{Lip} }R(g_{K-1}).
	\end{align}
	Note that 
	\begin{align*}
	\int F \ d(\nu_{\cup_{i=1}^{K-1}I_i}\otimes \nu_{I_K})  = \int \tilde{F}  d \nu_{\cup_{i=1}^{K-1}I_i},
	\end{align*}
	where $\tilde{F} : [-L, L]^{ |I| - |I|_K } \to \bR$ is defined by
	\begin{align*}
	\tilde{F}(y_1, \ldots,  y_{ p_{K-1} } ) =  \int_M F( y_1, \ldots, y_{ p_{K - 1 }  }   , X_{i_{p_{K-1}+1}}(x_K), \ldots ,X_{i_{p_{K}}}(x_K)) \, d \mu(x_K).
	\end{align*}
	Since $\tilde{F}$ is separately Lipschitz continuous with
	$$
	\Vert \tilde{F} \Vert_{\Lip} \le \Vert F \Vert_\Lip,
	$$
	it follows from the induction hypothesis that
	\begin{align}\label{eq:ih_fcb}
		\biggl| \int \tilde{F}  \, d\nu_{\cup_{i=1}^{K-1}I_i}
		- \int \tilde{F} \, d(\nu_{I_1}\otimes\cdots\otimes \nu_{I_{K-1}}) \biggr| 
		\le C_{\ast}\Vert F\Vert_{ \Lip } \sum_{i=1}^{K-2}R(g_{i}).
	\end{align}
	Combining \eqref{eq:fcb_base_case} and \eqref{eq:ih_fcb} yields 
	\eqref{eq:fcb2} for $K$, which completes the proof. 
\end{proof}

\begin{remark}  
Correlation decay conditions such as \eqref{eq:fcb} arise quite naturally in applications of Stein's method to distributional approximations of dynamical systems. Let us illustrate this in the context of normal approximations. In the case of the standard normal distribution $N(0,1)$, the Stein equation is given by
\begin{align}\label{eq:se_normal}
f'(w) - w f(w) = h(w) - \Phi(h),
\end{align}
where $\Phi(h)$ denotes the expectation of $h$ with respect to $N(0, 1)$.  
Let $(T, M, \cF, \mu)$ be a measure-preserving transformation, and let $\varphi : M \to \bR$ be a bounded measurable observable with $\mu(\varphi) = 0$. Consider the partial sum  
$
V = b^{-1} \sum_{n=0}^{N-1} X_n,
$  
where $X_n = \varphi \circ T^n$, $b =  \sqrt{\text{Var}_\mu(  \sum_{n=0}^{N-1} X_n    )}$, and it is assumed that $b > 0$.
Here, $T^n = T \circ T^{n-1}$ with $T^0$ being the identity map.  
Solving \eqref{eq:se_normal} for a given test function $h$ and taking expectations, we have  that 
$$
| \mu[h(V)] - \Phi(h) | \leq | \mu[f'(V) - V f(V)] |.
$$  
It is well known (see \cite{CGS11}) that for a Lipschitz continuous $h$, the solution $f$ to \eqref{eq:se_normal} has bounded 
derivatives up to second order. Introducing the punctured sums $V_{n,K} := \sum_{0 \leq i < N, |i - n| > K} X_i$, we can decompose  
$
\mu[f'(V) - V f(V)]
$ as follows:  
	\begin{align}
		&\mu[V f(V) - f'(V)] \notag \\
		&= b^{-1} \mu \biggl( \sum_{n=0}^{N-1} X_n (f(V) - f(V_{n,K}) - f'(V)(V - V_{n,K})) \biggr) \label{eq:1d_first} \\
		&\quad + \mu \biggl[ \biggl( b^{-1} \sum_{n=0}^{N-1} \sum_{ 0 \le m < N, \, |n-m| \le K  } X_n X_m - 1 \biggr) f'(V) \biggr] \label{eq:1d_second} \\
		&\quad + b^{-1} \mu \biggl( \sum_{n=0}^{N-1} X_n f(V_{n,K}) \biggr). \label{eq:1d_third}
	\end{align}  
	Controlling \eqref{eq:1d_first} and \eqref{eq:1d_second} involves Taylor expansions, along with mixing properties of the dynamics. 
	For simplicity, let us focus on \eqref{eq:1d_third}.  
	Observe that, if $(X_n)$ in \eqref{eq:1d_third} is replaced with a sequence of centered and independent random variables, then \eqref{eq:1d_third} vanishes as soon as $K \geq 1$. However, for a dependent sequence, \eqref{eq:1d_third} can be large. Using $\mu(X_n) = \mu(\varphi) = 0$, we can write 
	$$
	\mu  ( X_n f(V_{n,K}) ) = \int F \, d \nu_{I} - \int F \, d (\nu_{I_1} \otimes \nu_{I_2} \otimes \nu_{I_3}),
	$$  
	where  $I = \cup_{i=1}^3 I_i$ with 
	$$
	I_1 = \{0 \leq i < N : i < n - K\}, \quad I_2 = \{n\}, \quad I_3 = \{0 \leq i < N : i > n + K\},
	$$  
	and $F : [- \Vert \varphi \Vert_\infty, \Vert \varphi \Vert_\infty]^{|I|} \to \bR$ is a separately Lipschitz continuous function with $\Vert F \Vert_{\Lip} \leq C \Vert \varphi \Vert_\infty \Vert f \Vert_{\Lip}$ for some absolute constant $C > 0$.  
	Assuming \eqref{eq:fcb2}, we find that  
	$$
	|\eqref{eq:1d_third}| \le N b^{-1} C_* \Vert F \Vert_{ \Lip }  2 R(K + 1)
	\le 
	N b^{-1} C C_* \Vert \varphi \Vert_\infty \Vert f \Vert_{\Lip} 2 R(K + 1),
	$$  
	where it is natural to expect that $R(n)$ decays rapidly as $n \to \infty$ if $T$ is a suitably chaotic dynamical system and $\varphi$ is sufficiently regular.
\end{remark}

\subsection{Main result: Brownian approximation under \eqref{eq:fcb}}\label{sec:main_thm_statement}
Let $X_0, \ldots, X_{N-1}$ be real-valued random variables as in \eqref{eq:assume_rv}.
Define the following quantities:
\begin{align*}
	&\sigma_{i,j} = \mu(X_iX_j), 
	\quad 
	\beta_i = \sum_{j = 0}^{N-1} \sigma_{i,j},
	\quad 
	B_k = \sum_{i=0}^{k-1} \beta_i
	= \sum_{i = 0}^{k-1} \sum_{j=0}^{N-1} \sigma_{i,j},
\end{align*}
where we adopt the conventions $\beta_0 = B_0 = 0$. Note that 
\begin{align}\label{eq:def_B}
B := B_N = \text{Var}_\mu \biggl(  \sum_{i=0}^{N-1} X_i \biggr).
\end{align}
We assume that $B > 0$.  As noted in Section \ref{sec:appli},
in the case that $X_n$ is generated by a chaotic dynamical system, it typically holds that 
$$
\inf_{N\ge N_0}\tilde B>0,\quad \tilde B:=B/N
$$
with some integer $N_0$.

For $t \in [0,1]$ and $0 \le n < N$, let 
\begin{align}\label{eq:def_theta}
\theta_n(t) = J_{ B_n / B }(t),
\end{align}
and define a random element $S_N \in D$ by 
\begin{align*}
S_N(t) &= \sum_{n=0}^{N-1}   \theta_n(t)  X_n.
\end{align*}

Let $Z$ be a standard Brownian motion. By adapting Barbour's approach \cite{B90}, we establish the following error bound in approximating the law of 
\begin{align}\label{eq:def_W}
	W_N = B^{-1/2} S_N
\end{align}
by the law of $Z$ with respect to an integral probability metric of smooth test functions.

\begin{thm}\label{thm:main}
	Assume that $(X_n)_{0 \le n < N}$ satisfies the functional correlation bound 
	\eqref{eq:fcb}
	with rate function $R$ and constant $C_*$. 
	Then, for any
	$g \in \sM_0(C_0)$,
	\begin{align}\label{eq:bound_main}
		| \mu [ g(W_N) ] - \bfE[g(Z)] |  \le C  C_{N} \Vert g \Vert_{\sM} B^{-3/2} N  = C  C_{N} \Vert g \Vert_{\sM} \tilde{B}^{-3/2}  N^{-1/2},
	\end{align}
	where $C > 0$ is an absolute constant,
	\begin{align*}
		C_{N} = (L+1)^3 \biggl\{  1 + ( C_*  + 1)  \hat{R}_3(N)
		+ C_0C_*^{3/2}    \hat{R}_1(N) ^{1/2}( 1 +  \hat{R}_2(N)) \biggr\},
	\end{align*}
	and
	$$
	\hat{R}_1(N) = \sum_{j=0}^{N-1} \hat{R}(j), \quad \hat{R}_2(N) = \sum_{j=0}^{N-1} j^{1/2} \hat{R}(j), \quad 
	\hat{R}_3(N) = \sum_{j=0}^{N-1} j \hat{R}(j)
	$$
	with $\hat{R}(n) = \max_{n \le j \le N} R(j)$ for $n \ge 1$ and $\hat{R}(0) = 1$.
\end{thm}

\begin{remark} Notice that,
	\begin{align*}
		C_{N} = C (L+1)^3 ( C_0 + 1 ) (C_* + 1)^{3/2} ( \hat{R}_3(N) + 1  )^{3/2},
	\end{align*}
	where $C > 0$ is an absolute constant.
	Hence, if
	$(X_n)_{n \ge 0}$ is a sequence of real-valued random variables such that, for 
	all $N \ge 1$, 
	\eqref{eq:fcb} holds with  
	a rate function $R$ that satisfies $\sum_{m=1}^\infty m \hat{R} (m) < \infty$, and if $B^{-1/2} = O(N^{-1/2})$, then we obtain 
	$| \mu [ g(W_N) ] - \bfE[g(Z)] | = O(N^{-1/2})$ for all $g \in \sM_0$.
\end{remark}

\begin{remark}
We will establish \eqref{eq:bound_main} under a set of correlation decay conditions weaker than \eqref{eq:fcb}. Specifically, Theorem \ref{thm:prelim} shows that, to obtain \eqref{eq:bound_main} for a particular test function $g \in \sM_0$, it suffices to verify \eqref{eq:fcb} for a subclass of separately Lipschitz functions associated with $g$. However, the expressions for the functions in this subclass are complicated, resulting in conditions that are seemingly more technical than \eqref{eq:fcb}.
\end{remark}

\subsection{About weak convergence} 
As the family of test functions $\sM_0$ is significantly smaller than the collection of all bounded Skorokhod-continuous functions, convergence of the right hand side of \eqref{eq:bound_main} to zero
does not automatically yield the weak convergence $W_N \Rightarrow Z$
with respect to the Skorokhod topology. On the other hand, it follows from
\cite[equation (1.16)]{BRZ24} that, provided the right hand side of \eqref{eq:bound_main} converges to zero 
as $N \to \infty$, then $W_N \Rightarrow Z$ if for each $\lambda > 0$,
\begin{align}\label{eq:weak_cond}
	\limsup_{\ve \to 0} \limsup_{N \to \infty} \mu \biggl( \Vert ( W_N )_\ve - W_N    \Vert_\infty  \ge \lambda \biggr) = 0,
\end{align}
where $w_\ve$ for $w \in D$ is defined as in \eqref{eq:w_eps}.
Using \cite[Lemma 1.4]{BRZ24}, we obtain 
as a consequence of \eqref{eq:weak_cond} and \eqref{eq:bound_main}
the following sufficient condition for $W_N \Rightarrow Z$, suitable for applications to dynamical systems. For example, it follows that $W_N \Rightarrow Z$ holds 
in Theorem \ref{thm:pikovsky} if $\gamma < 3/2$, and in Theorem \ref{thm:billiards}.

\begin{prop}\label{prop:weak} Let $(X_n)_{n \ge 0}$ be a sequence of variables as in \eqref{eq:assume_rv} satisfying the 
	following conditions:
	\begin{itemize}
		\item[(i)] For all $N \ge 1$, \eqref{eq:fcb} holds with a rate function $R$ such that 
		$\sum_{n=1}^\infty n R(n) < \infty$ and  
		$\sum_{n=1}^\infty R(n)^{1/2} < \infty$.
		\smallskip
		\item[(ii)] $B_N^{-1} = o(N^{-2/3})$. \smallskip
		\item[(iii)] There exist $C_1 > 0$,  
		$N_0 \ge 1$ and $M_N \ge 1$ with $M_N = o(B_N)$
		such that, whenever $N \ge N_0$,
		\begin{align*}
			\inf_{ 0 \le u,v < N, \:  v - u \ge M_N } ( B_v^{ (N) } - B_u^{(N)}   )  \ge C_1 M_N,
		\end{align*}
		where $B_u^{(N)}  = \sum_{i=0}^{u-1} \sum_{j=0}^{N-1} \mu(X_iX_j)$.
	\end{itemize}
	Then, $W_N$ converges weakly to $Z$ with respect to the Skorokhod topology as 
	$N \to \infty$.
\end{prop}

\begin{remark}
	Assume that (i) holds. Then a sufficient condition for both (ii) and (iii) is:
	\begin{itemize}
		\item[(ii')] There exist $\eta \in (0,1]$ and $\Sigma > 0$ such that 
		$
		| N^{-1} B_N - \Sigma | = O(N^{-\eta}).
		$
	\end{itemize}
	To derive (iii) from (ii'), let $0 \le u < u + h < N$
	and denote $S_u = \sum_{i=0}^{u-1} X_i$. Since $\sum_{n=1}^\infty nR(n) < \infty$, 
	$$
	B_{u+h}^{(N)} - B_u^{(N)} \;\ge\; \mu(S_{u+h}^2) - \mu(S_u^2) - C
	$$
	for some constant $C>0$ independent of $u,h,N$.  
	From (ii') we have
	$$
	\mu(S_{u+h}^2) - \mu(S_u^2) 
	\ge h \Sigma - C' |u+h|^{\,1-\eta} 
	\ge h \Big( \Sigma - C' h^{-1} (2N)^{1-\eta} \Big)
	$$
	for some $C'>0$ independent of $u,h,N$.  
	Choosing $M_N =  \round{ C_2 N^{1-\eta} }$ 
	for $C_2$ sufficiently large
	gives (iii).  
	In particular, if $X_n = f \circ T^n$ where $T : M \to M$ is $\mu$-preserving and 
	$f : M \to \mathbb{R}$ is bounded with $\mu(f) = 0$, then (ii) and (iii) hold provided 
	$$
	\Sigma = \mu(f^2) + 2 \sum_{n=1}^\infty \mu(f \circ T^n \cdot f) > 0,
	$$
	since in this case (ii') holds with $\eta = 1$.
\end{remark}

\begin{proof}
	In what follows we will denote 
	by $C$ a constant whose exact value is unimportant, and whose value can change from one occurence
	to the next. 
	
	By (i) and (ii), the right hand side of \eqref{eq:bound_main} tends to zero as $N \to \infty$. Therefore, 
	by \cite[equation (1.16)]{BRZ24}, it suffices to verify \eqref{eq:weak_cond}. We do this by applying \cite[equation (1.23)]{BRZ24} together with a bound on the fourth moment 
	$\mu[ ( W_N(t) - W_N(s) )^4 ]$ derived using (i)-(iii).
	
	Let $N \ge N_0$.
	For $0 \le s \le t \le 1$ define 
	$$
	I_N(s,t) =  \{ 0 \le k < N \: : \:  B_k^{(N)} / B_N \in (s,t]  \}.
	$$
	It follows from (iii) that
	\begin{align}\label{eq:I_estim}
		|I_N(s,t)| \le C M_N +  C B_N ( t- s).
	\end{align}
	To see this, let $k_1 < \ldots < k_m$, $m \ge 1$, be an enumeration of the elements of $I_N(s,t)$ in increasing order.
	Set 
	$$
	q = \biggl\lfloor \frac{m - 1}{M_N}  \biggr\rfloor, \quad r_j = k_{ 1 + jM_N }, \quad 0 \le j \le q.
	$$
	If $q = 0$ we have $m < M_N + 2 \le 3 M_N$. So, we may assume $q \ge 1$. Then, 
	since $r_{j+1} - r_j \ge M_N$ and $B_k^{(N)} / B_N \in (s,t]$ for $k \in I_N(s,t)$, (iii) yields
	\begin{align*}
		C_1 q M_N   \le 
		\sum_{j=1}^q ( B_{ r_j }^{(N)} - B_{r_{j-1}}^{(N)} )
		=
		B_{  r_q }^{(N)} - B_{r_0}^{(N)} \le B_N(t - s),
	\end{align*}
	from which it is straightforward to deduce \eqref{eq:I_estim}. It is then a direct verification using 
	$\sum_{n=1}^\infty R(n)^{1/2} < \infty$ that, with the convention $R(0) = 1$, 
	\begin{align}\label{eq:fourth_moment}
		\begin{split}
			\mu[  ( W_N(t ) - W_N(s)  )^4  ] &\le B_N^{-2} \sum_{ i \in I_N(s,t) }  \sum_{ j \in I_N(s,t) }  \sum_{ k \in I_N(s,t) }  \sum_{ \ell \in I_N(s,t) } |\mu( X_i X_j X_k X_\ell )| \\
			&\le C B_N^{-2} \sum_{  i \le j \le k \le \ell , \:  i,j,k,\ell \in I_N(s,t)  }  \sqrt{ R( j - i ) R( \ell - k )  }  
			\\
			&\le C B_N^{-2} \biggl(   \sum_{ i \le j, \:   i,j \in I_N(s,t)  }  \sqrt{ R(j - i) }   \biggr)^2 \\
			&\le  C ( |t - s|^2 + ( M_N / B_N )^2  ).
		\end{split}
	\end{align}
	Consequently,
	\begin{align}\label{eq:tightness}
		\mu[  ( W_N(t ) - W_N(s)  )^4  ] \le C|t-s|^2
	\end{align}
	holds if $t - s \ge 2^{-1} K_N^{-1}$ with $K_N = \round{B_N / M_N}$,
	where $K_N \to \infty$ by (ii) and (iii).
	Applying Markov's inequality, we see that \cite[equation (1.24)]{BRZ24}
	holds with $\beta = 2$ and $n = K_N$.
	By another application 
	of \eqref{eq:fourth_moment} we find that for sufficiently large $N$ and any 
	$s \in [ (k-1) / K_N,  k / K_N ]$,
	\begin{align}\label{eq:varphi_n}
		\begin{split}
		K_N \mu( |  W_N(s)  - W_N( (k-1) / K_N )    |^4 ) 
		&\le C K_N \biggl(  K_N^{-2}  +  ( M_N / B_N )^2  \biggr) \\
		&\le C M_N / B_N \to 0.
		\end{split}
	\end{align}
	Hence, the quantity in \cite[equation (1.22)]{BRZ24} converges to zero.
	Now, \eqref{eq:tightness}, \eqref{eq:varphi_n}, and \cite[equation (1.23)]{BRZ24} imply
	\begin{align}\label{eq:approx_w}
		\mu \biggl( \Vert ( W_N )_\ve - W_N    \Vert_\infty  \ge \lambda \biggr) 
		\le C_\lambda (   M_N / B_N + \ve  ),
	\end{align}
	whenever $\ve \in (  1 / K_N, 1 )$, where $C_\lambda > 0$ is a constant independent of $\epsilon$.
	In particular, \eqref{eq:weak_cond} holds.	
\end{proof}

\begin{remark}
Combining \eqref{eq:approx_w} and 
\eqref{eq:bound_main} with the Gaussian smoothing inequalities in \cite[Theorem 1.1 and Corollary 1.3]{BRZ24}, 
the error bounds in \eqref{eq:approx_w} can be extended
to test functions with less regularity than those 
in $\sM_0$. For example, substituting \eqref{eq:approx_w} 
and \eqref{eq:bound_main} into \cite[equation (1.19)]{BRZ24} yields together 
with \cite[equation (1.28)]{BRZ24} a rate of convergence with respect to the 
Prokhorov metric after optimizing for the parameters in a manner similar to \cite[example 1.8]{BRZ24}.
The rates obtained in this way are, however, considerably slower than those given by 
\eqref{eq:bound_main} for functions $g \in \mathscr{M}_0$. 
A detailed treatment of these applications is beyond the scope of this paper.
\end{remark}

\section{Applications}\label{sec:appli}

In this section, we apply Theorem \ref{thm:main} to four concrete examples of dynamical systems: (1) Pikovsky maps, (2) Liverani--Saussol--Vaienti maps, (3) Lasota--Yorke maps, and (4) dispersing Sinai billiards. Among these, our primary contribution lies in establishing a functional correlation bound for Pikovsky maps, which is done in Section \ref{sec:proof_pikovsky}. For the other models, Theorem \ref{thm:main} is directly applicable due to functional correlation bounds that have already been established in prior works: \cite{L17} for model (2), 
\cite{LS20} for model (3), and \cite{LS17} for model (4).
Additionally, for models (2) and (3), we give results for nonautonomous compositions -- sequential compositions for (2) and random compositions for (3) -- these cases also follow directly from earlier results.

\subsection{Application 1: Pikovsky maps} Following \cite{P91}, for $\gamma > 1$, let $T :  M \to M$ be a map on $M = [-1,1]$ 
with graph as in Figure \ref{fig:maps},  defined implicitly by the relations
\begin{align}\label{eq:pikovsky}
	x = \begin{cases}
		\frac{1}{2 \gamma } ( 1 + T(x))^{\gamma},  &0 \le x \le \frac{1}{2 \gamma }, \\
		T(x) + \frac{1}{2 \gamma } (1 - T(x))^{\gamma}, &\frac{1}{2 \gamma } \le x \le 1,
	\end{cases}
\end{align}
and by letting $T(x) = - T(-x)$ for $x \in [-1,0]$.
The map has neutral fixed points
at $x = 1,-1$, while at $x = 0$ its derivative becomes infinite. The parameter $\gamma > 1$ controls 
both the neutrality of the fixed points and the degree of the singularity:
\begin{align*}
	&T'(x) \approx C_\gamma x^{ 1/ \gamma - 1 }   \quad \text{as $x \downarrow 0$}, \\
	&T'(x) \approx 1 +  \frac12 (1 - x)^{ \gamma - 1 }  \quad \text{as $x \uparrow 1$},
\end{align*}
where $C_\gamma > 0$ is a constant depending on $\gamma$, and $A \approx B$ indicates
$A/B \to 1$.
Note that $T$ reduces to the angle doubling map at the limit $\gamma \downarrow 1$.

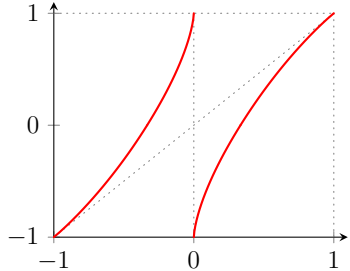
\begin{figure}[h!]
	\begin{tikzpicture}[every node/.style={scale=0.8}]
		\tikzmath{
			\ggamma = 1.5;
		}
		\begin{axis}[
			height       = 1.85in,
			xmin         = -1.0,
			xmax         = 1.1,
			ymin         = -1.0,
			ymax         = 1.1,
			xtick        = {-1.0, 0.0, 1.0},
			xticklabels  = {$-1$, $0$, $1$},
			axis x line  = bottom,
			ytick        = {-1.0, 0.0, 1.0},
			yticklabels  = {$-1$, $0$, $1$},
			axis y line  = left,
			line cap=round
			]
			% dashed lines
			\addplot[gray,dotted] coordinates { (-1,1) (1,1) };
			\addplot[gray,dotted] coordinates { (1,1) (1,-1) };
			\addplot[gray,dotted] coordinates { (-1,-1) (1,1) };
			\addplot[gray,dotted] coordinates { (0,-1) (0,1) };
			% Graph of T
			\addplot[red,thick,domain=-1.0:0.0, samples=21]({0.5 / \ggamma * (1 + x)^\ggamma}, {x});
			\addplot[red,thick,domain=0.0:1.0, samples=21]({x + 0.5 / \ggamma * (1 - x)^\ggamma}, {x});
			\addplot[red,thick,domain=-1.0:0.0, samples=21]({- 0.5 / \ggamma * (1 + x)^\ggamma}, {-x});
			\addplot[red,thick,domain=0.0:1.0, samples=21]({- x - 0.5 / \ggamma * (1 - x)^\ggamma}, {-x});
		\end{axis}
	\end{tikzpicture}
	\caption{Graph of the Pikovsky map~\eqref{eq:pikovsky}}
	\label{fig:maps}
\end{figure}

Let $\lambda$ denote the Lebesgue measure on $M$ normalized to probability. It follows 
from \eqref{eq:pikovsky} that the transfer operator $\cL_T : L^1(\lambda) \to L^1(\lambda)$ associated with 
$(T, \lambda)$, defined by the duality relation
$$
\int_0^1 g \circ T \cdot f \, d \lambda =  \int_0^1 g \cdot  \cL_T(f) \, d \lambda \quad \forall f \in L^1(\lambda), \, \forall g \in L^\infty(\lambda),
$$
fixes the constant function $1$, so that 
$\lambda$ is 
invariant under $T$. (For properties of transfer operators of non-singular measurable maps, we refer to \cite{B00,B18}.) 
Given a function $f : M \to \bR$, we define $$X_n = f \circ T^n,$$ 
where $T^n = T^{n-1} \circ T$ for $n \ge 1$ and $T^0$ is the identity map. We 
consider 
$(X_n)$ as a random process on $(M, \cF, \mu)$, where $\cF$ is 
the $\sigma$-algebra of Borel sets and $\mu = \lambda$. 
By \cite[Proposition 5]{CHMV10}, we have 
\begin{align}\label{eq:corr_decay}
\mu( f \cdot f \circ T^n )  - \mu(f) \cdot \mu (f)  = O(n^{ - 1 / ( \gamma - 1) }),
\end{align}
whenever $f$ is H\"{o}lder continuous.
By  \cite[Proposition 7]{CHMV10}, 
$(X_n)$ satisfies a central limit theorem if $\gamma < 2$, and the associated 
FCLT follows from \cite{CHMV10, CLM23, MN05}.

\begin{lem}\label{lem:fcb_pikovsky} 
	Let $\gamma > 1$, and let $f$ be Lipschitz continuous. Then, 
	for all $N \ge 1$, 
	\eqref{eq:fcb} holds with 
	$$
	C_* =( \Vert f \Vert_{ \Lip } + 1 )C_\gamma, \quad R(n) = n^{ - 1 / ( \gamma - 1) }, 
	$$
	where $C_\gamma$ is a constant depending only on $\gamma$. 
\end{lem}

We postpone proving Lemma \ref{lem:fcb_pikovsky} until Section \ref{sec:proof_pikovsky} and proceed to describe an application of Theorem \ref{thm:main}. The quantities $W_N$, $B$, and $Z$ appearing below 
are defined as in Section \ref{sec:main_thm_statement}.

\begin{thm}\label{thm:pikovsky} Assume that $f : M \to \bR$ is Lipschitz continuous with
	$\mu(f) = 0$. If $\gamma < 2$ and $f$ is not a coboundary\footnote{We say that 
		$f$ is a coboundary if there exists $g \in L^2(\mu)$ such that $f = g - g \circ T$. It is known that for nonuniformly expanding Markov maps such as the Pikovsky map, the 
		non-coboundary condition is a generic condition~\cite{G06}.}\label{f:1219},
	$\tilde{B} := \inf_{N \ge 1}  N^{-1} B_N > 0$.
	Moreover,
	for any $g \in \sM_0(C_0)$ and $N \ge 1$,
	\[
	| \mu [ g(W_N) ] - \bfE[g(Z)] |  \le 
	C C_\#  \Vert g \Vert_\sM 
	( \rho _\gamma (N)   
	+    1  )  \tilde{B}^{-3/2}   N^{-1/2},
	\]
	where $C > 0$ is an absolute constant, 
	\[
	C_\# =  (C_0 + 1) (  1 +   \tfrac{\gamma - 1}{2 - \gamma} )^{3/2} (C_\gamma  + 1)^{3/2} ( \Vert f \Vert_{\Lip } +1)^{9/2} 
	\]
	with $C_\gamma$ is as in Lemma \ref{thm:pikovsky}, and 
	\[
	\rho _\gamma (N)
	=
	\begin{cases}
		1 - c_\gamma^{-1} \quad &(\gamma <3/2)\\
		1+\log N \quad &(\gamma =3/2)\\
		1 + c_\gamma^{-1} N^{  c_\gamma } \quad &(\gamma >3/2)
	\end{cases}
	\]
	with $c_\gamma
	=  ( 2 \gamma - 3) / ( \gamma - 1)$.
	In particular, $| \mu [ g(W_N) ] - \bfE[g(Z)] | $ converges to zero as $N \to \infty$ when $\gamma< 5/3$.
\end{thm}

\begin{proof} Since $\gamma < 2$, we have
	\begin{align}\label{eq:Sigma}
	\lim_{N \to \infty} N^{-1} B_N = \lim_{N \to \infty} N^{-1} \text{Var}_\mu \biggl( \sum_{n=0}^{N-1} X_n   \biggr)
	= \mu(f^2) + 2 \sum_{n=1}^\infty \mu( f \cdot f \circ T^n ) =: \Sigma,
	\end{align}
    where the series converges by 
    \eqref{eq:corr_decay}. Moreover, 
    $\Sigma = 0$ if and only if $f$ is 
    a coboundary; see \cite{L95, G06, KKM18}.
	Therefore, there exists $N_0 \ge 1$ such that 
	$N^{-1} B_N \ge \Sigma$ for all $N \ge N_0$. On the other hand, we must have $B_N / N > 0$ if $1 \le N \le N_0$.  
	Indeed, if $B_N / N = 0$ for some $N$, then 
	$\sum_{n=0}^{N-1} X_n$ is constant $\mu$-a.e., and by stationarity it follows that, 
	for each $k \ge 1$, $B_{kN} = \mu( ( \sum_{n=0}^{kN - 1} X_n )^2  ) = 0$.  
	This contradicts $\lim_{N\to\infty} B_N / N = \Sigma > 0$.
	An application of 
	Theorem \ref{thm:main} combined with Lemma \ref{lem:fcb_pikovsky} now yields the upper bound	
	\begin{align*}
		&| \mu [ g(W_N) ] - \bfE[g(Z)] |  \\
		&\le C C_\# \Vert g \Vert_{\sM} 
		( 1 + \hat R_3(N) )(1+\hat R_1(N))^{3/2}  
		\Vert g \Vert_{\sM} \tilde{B}^{-3/2}  N^{-1/2} \\
		&\le C C_\# \Vert g \Vert_{\sM} 
		\biggl(1 + \sum_{n=0}^{N-1} n^{- 1 / (\gamma - 1)} \biggr)^{3/2}  
		\biggl( 1 + \sum_{n=0}^{N-1} n^{  1 - 1 / (\gamma - 1)  } \biggr) \Vert g \Vert_{\sM} \tilde{B}^{-3/2}  N^{-1/2}.
	\end{align*}
	Further,
	\[
	\sum_{n=1}^{N-1} n^{1 - 1/(\gamma - 1) }
	\le 
	\begin{cases}
		1 +  (\gamma - 1) / ( 3 - 2 \gamma ) \; &(  \gamma < 3 /2 )\\
		1 + \log N\; &( \gamma  = 3/2 )\\
		1 - (\gamma - 1) / ( 3 - 2 \gamma ) N^{ 2 - 1 / (  \gamma - 1 ) } \; &( \gamma > 3/2),
	\end{cases}
	\]
	and $\sum_{n=1}^{N-1} n^{ - 1 / ( \gamma - 1) } \le  1 + ( \gamma - 1 ) / ( 2 - \gamma )$.
\end{proof}

\subsection{Application 2: (Non)autonomous Liverani--Saussol--Vaienti maps}
For \(\alpha \in (0,1)\), consider the map  \(T : M \to M\) on $M = [0,1]$ defined by 
\begin{align}\label{eq:lsv}
	T(x) = \begin{cases} x(1+ 2^{\alpha }x^{\alpha}) &  x \in [0, 1/2),\\
		2x-1 & x \in [1/2,1].
	\end{cases} \hspace{0.5cm}
\end{align} 
Let $\beta \in (0,1)$ and $a_\beta > 2^\beta(\beta+ 2)$.
Following~\cite{LSV99},
we define
the convex and closed cone
\begin{align*}
	\cC_*(\beta) = \{f\in C((0,1])\cap L^1(\lambda) \,:\,
	& \text{$f\ge 0$, $f$ decreasing,}
	\\
	& \text{$x^{\beta+1}f$
		increasing, $f(x)\le a_\beta x^{-\beta} \lambda(f)$}\},
\end{align*}
where $\lambda$ denotes the Lebesgue measure on $M$.
It follows from~\cite{LSV99, AHNTV15} that $\cC_*(\beta)$ contains 
densities of invariant absolutely continuous probability measures
for maps~\eqref{eq:lsv} with parameters $\alpha \in (0, \beta]$.

\subsubsection{Nonautonomous compositions} 
We begin by giving a result for nonautonomous compositions of maps in \eqref{eq:lsv} with time-dependent parameters. From this, we deduce a corresponding result for 
conventional measure preserving systems. We suppose that $T_1, T_2, \ldots$ 
are maps 
in \eqref{eq:lsv} with corresponding parameters $\alpha_1, \alpha_2, \ldots$, and that $f_n : M \to \bR$ 
are functions such that:
\begin{align}\label{eq:ass_lsv}
	\alpha_* := \sup_{n \ge 1} \alpha_n < 1/2, \quad  
	\sup_{n \ge 1} \Vert f_n \Vert_{ \Lip } \le L, 
	\quad \text{and} \quad \mu(f_n \circ T_n \circ \cdots \circ T_1) = 0 \quad \forall n \ge 0,
\end{align}
where $\mu$ is a probability measure whose density belongs to $\cC_*(\alpha_*)$. We define 
\begin{align}\label{eq:time-dep}
	X_n = f_n \circ T_{1,n}, \quad T_{1,n} = T_n \circ \cdots \circ T_1,
\end{align}
and 
consider $(X_n)_{n\ge 0}$ as a random 
process on $( M, \cF, \mu )$.
Central limit theorems and associated weak/strong invariance principles for nonstationary processes generated by intermittent maps have been studied 
in several previous works including \cite{S22, NPT21, L18, NTV18}.
Here, we complement these results by providing an error bound in a self-normed FCLT.

\begin{lem}\label{lem:fcb_lsv} Assume \eqref{eq:ass_lsv}.
	Then, for all $N \ge 1$, \eqref{eq:fcb} holds with
	\begin{align*}
		C_* = C_{\alpha_*}(L+1), \quad R(n) = \begin{cases}
			n^{ 1 - 1 / \alpha_* } ( \log n )^{ 1 / \alpha_* }, &\text{if $n \ge 2$}, \\
			1, &\text{if $n = 1$,}
		\end{cases}
	\end{align*}
	where $C_{\alpha_*} > 0$ is a constant depending only on $\alpha_*$.
\end{lem}

\begin{proof} Let $F$, $I, I_1,I_2$ be as in Definition \ref{defn:fcb}. Then, 
	\begin{align}\label{eq:rewrite_fcb}
		\begin{split}
			&\int F \, d \nu_{I} 
			- \int F d ( \nu_{I_1} \otimes  \nu_{I_2} ) \\
			&= \int_M F(  X_{i_1}(x), \ldots, X_{i_n}(x)  ) \, d \mu(x) \\
			&- \iint_{M^2} 
			F(  X_{i_1}(x), \ldots, X_{i_{p_1}}(x), X_{i_{p_1 + 1}  }(y), \ldots, X_{i_n}(y)  ) \, d \mu(x) \, d \mu(y) \\
			&= \int_M  \tilde{F} (  T_{1, i_1}(x), \ldots, T_{1, i_n}(x)  ) \, d \mu(x) \\
			&- \iint_{M^2} 
			\tilde{F} (  T_{1, i_1}(x), \ldots, T_{1, i_{p_1}}(x), T_{1, i_{p_1 + 1 }  }(y), \ldots, T_{1, i_n}(y)  ) \, d \mu(x) \, 
			d \mu(y),
		\end{split}
	\end{align}
	where 
	$$
	\tilde{F}(  x_{1}, \ldots, x_{n}  ) = F( f_{  i_1  }(x_{1}), \ldots, f_{  i_n  }(x_{n}) ).
	$$
	Since $\Vert \tilde{F} \Vert_{\Lip} \le \Vert F \Vert_{\Lip} (L + 1)$, \eqref{eq:fcb}
	corresponds to a special case of \cite[Theorem 1.1]{L17}, which implies the upper bound 
	$$
	\biggl| \int F \, d \nu_{I} 
	- \int F d ( \nu_{I_1} \otimes  \nu_{I_2} ) \biggr| \le C_{\alpha_*} (L+1)  \Vert F \Vert_{\Lip} R(i_{ p_1 + 1 } - i_{p_1}),
	$$
	where $C_{\alpha_*} > 0$ is a constant depending only on $\alpha_*$ and $R(n) = n^{ 1 - 1 / \alpha_* } ( \log n )^{ 1 / \alpha_* }$ for $n \ge 2$ and $R(1) = 1$, as 
	desired.
\end{proof}

\begin{thm}\label{thm:lsv_nonstat}
	Let $N \ge 1$.
	Assume that, in addition to \eqref{eq:ass_lsv}, $B > 0$.
	Then, there exists a constant $C_{\alpha_*} > 0$ depending only on $\alpha_*$ such that for any $g \in \sM_0(C_0)$,
	\begin{align}\label{eq:1227}
		| \mu [ g(W_N) ] - \bfE[g(Z)] |  \le 
		C_{\alpha_*}  C_\# \Vert g \Vert_\sM 
		( \rho _\gamma (N)   
		+    1  )  B^{-3/2}   N,
	\end{align}
	where $C_\# = (C_0 + 1)( L + 1 )^{9/2}$ and 
	\begin{align}\label{eq:rho_lsv}
		\rho(N)=
		\begin{cases}	
			1\quad  &(\alpha_* < 1/3)\\
			1 + (\log N)^4 \quad  &(\alpha_*= 1/3)\\
			N^{ 3 - 1/\alpha_* } (  \log N )^3 \quad  &(\alpha_* > 1/3)
		\end{cases}
	\end{align}
	In particular,  $| \mu [ g(W_N) ] - \bfE[g(Z)] | $ converges to zero as 
	$N \to \infty$ provided that $B \gg (\rho(N)  N )^{2/3}$. Further, assuming 
	$B^{-1} = O(N^{-1})$, convergence of $| \mu [ g(W_N) ] - \bfE[g(Z)] | $
	 to zero holds if $\alpha_* < 2/5$.
\end{thm}

\begin{proof}
	The result follows by applying Theorem \ref{thm:main} together with Lemma \ref{lem:fcb_lsv} as in the proof of Theorem \ref{thm:pikovsky}. 
\end{proof}

\subsubsection{Autonomous compositions} Suppose that $X_n = f \circ T^n$, where $T$ denotes the map \eqref{eq:lsv} with parameter $\alpha < 1/2$,
and $f : [0,1] \to \bR$ is  Lipschitz continuous.  Suppose that  $\mu(f) = 0$, 
where $\mu$ is the unique absolutely continuous invariant probability measure of $T$  (see \cite{LSV99} for the existence of such a measure).
This setting is a special 
case of  \eqref{eq:time-dep}, obtained by taking $f_n = f$, $T_n = T$ for each $n \ge 1$, and $\mu$ as the invariant measure. 
As a  consequence Theorem \ref{thm:lsv_nonstat} we obtain the following result:

\begin{cor} Assume that $f$ is not a coboundary. Then, $\tilde{B} := \inf_{N \ge 1}  N^{-1} B_N > 0$. Moreover, for $N \ge 1$,
\begin{align*}
	| \mu [ g(W_N) ] - \bfE[g(Z)] |  \le 
	C_{\alpha} C_\#  \Vert g \Vert_\sM 
	( \rho _\gamma (N)   
	+    1  )  \tilde{B}^{-3/2}   N^{-1/2} ,
\end{align*}
where $C_\alpha$ is a constant depending only on $\alpha$, $C_\# = (C_0 + 1)(  \Vert f \Vert_{ \Lip } + 1 )^{9/2}$, and $\rho(N)$ is defined as in 
\eqref{eq:rho_lsv} for $\alpha_* = \alpha$.
\end{cor}

\begin{proof} As in the proof of Theorem \ref{thm:lsv_nonstat} we find that $\tilde{B} > 0$ provided 
that $\alpha < 1/2$. Thus, the result follows directly from Theorem \ref{thm:lsv_nonstat}.
\end{proof}

\subsection{Application 3: Random Lasota--Yorke maps}

We consider random dynamical systems of piecewise expanding interval maps, as studied in \cite{DFGV18, B99} among others. To specify the model, let $\cE$ denote the collection of all maps $T : M \to M$ on $M = [0,1]$ such that there exists a finite partition $\cA(T)$ of $M$ into subintervals satisfying the following conditions for every $I \in \cA(T)$:

\begin{itemize}
	\item[(1)] $T |_I$ extends to a $C^2$ map in a neighborhood of $I$; \smallskip
	\item[(2)] $\delta(T) := \min_{I\in\cA(T)}\inf |(T\vert _I)'| > 1.$
\end{itemize}
The map $T$ is monotonic on each element $I \in \cA(T)$. From now on, we take 
$\cA(T)$ to be the minimal partition of monotonicity and define $n_0(T) := |\cA(T)|$.

Let $(\Omega, \cB, \bP)$ be a probability space, and let $\tau : \Omega \to \Omega$ be a measurably invertible transformation.
We consider a map $\omega \mapsto T_{\omega}$ from $\Omega$ into $\cE$. Random compositions of maps are denoted by
$
T^n_{\omega} = T_{\tau^{n-1}\omega} \circ \cdots \circ T_{\omega}
$.
Moreover, let $\cL^n_{\omega} = \cL_{\tau^{n-1}\omega} \cdots \cL_{\omega}$, where $\cL_\omega : L^1(\lambda) \to L^1(\lambda)$ is the transfer operator associated with $T_\omega$ and $\lambda$, with $\lambda$ being the Lebesgue measure on $M$. We assume the following conditions.

\noindent\textbf{Conditions (RLY):}
\begin{itemize}
	\item[(i)]{$\tau : \Omega \to \Omega$ is $\bP$-preserving and ergodic.} \smallskip
	\item[(ii)]{The map $(\omega, x) \mapsto (\cL_{\omega} H(\omega, \cdot) )(x)$ is measurable for every measurable function $H : \Omega \times M \to \bR$ such that $H(\omega, \cdot) \in L^1(\lambda)$. } \smallskip
	\item[(iii)]{$n_0 := \sup_{\omega \in \Omega} n_0(T_\omega) < \infty$; $\delta : =  \inf_{\omega \in \Omega} \delta(T_\omega) > 1$; $D := \sup_{\omega \in \Omega} |T''_\omega| < \infty$.
	} \smallskip
	\item[(iv)]{There exists $n \ge 1$ such that $\delta^n > 2$ and $\text{ess inf}_{\omega \in \Omega} \min_{J \in \cA(T_\omega^n)} \lambda(J) > 0$.} \smallskip
	\item[(v)]{For every subinterval $J \subset M$ there is $k = k(J) \ge 1$ such that $T^k_\omega(J) = M$ holds for almost every $\omega \in \Omega$.}
\end{itemize}

By \cite[Proposition 1]{DFGV18}, assuming (RLY),  there exists a unique measurable and nonnegative function $h : \Omega \times M \to \bR$ such that \(h_\omega := h(\omega, \cdot)\) satisfies \(\lambda (h_\omega) = 1\), 
and \(\cL_\omega(h_\omega) = h_{\tau(\omega)}\) for $\bP$-almost every \(\omega \in \Omega\). Moreover, \(\textnormal{ess sup}_{\omega \in \Omega} \|h_{\omega}\|_{BV} < \infty\).
Here, 
$\Vert h \Vert_{BV} = \Vert h \Vert_{L^1(m)} + V(  h)$ where 
$V( h )$ denotes the total variation of a function $h: M \to \bR$.
In addition, the skew product $\varphi(\omega, x) = (\tau(\omega), T_{\omega}(x))$ preserves the measure \(\nu\) on \(\Omega \times M\), defined by 
\[
\nu(A \times B) = \int_{A \times B} h \, d(\bP \otimes m), \quad A \in \cB, \, B \in \cF, 
\]
where \(\cF\) is the Borel \(\sigma\)-algebra on \(M\).
Denote by $(\nu_\omega)_{\omega\in\Omega}$ the disintegration of $\nu$.

\begin{lem}\label{lem:fcb_LY} Let $(f_{n,\omega})_{n \ge0, \omega \in \Omega}$ be a family of functions $f_{n,\omega} : M \to \bR$
	such that, for all $n \ge 0$, 
	$$\textnormal{ess sup}_{\omega \in \Omega} \Vert f_{n, \omega} \Vert_\vartheta
	\le  L$$
	for some $\vartheta\in (0,1]$. Consider the random process $(X_n)_{n \ge 0}$ on $(M, \cF, \nu_\omega)$ defined by 
	\[
	X_n= f_{n, \omega} \circ T_\omega^n.
	\]
	Then, for all $N \ge 1$ and for 
	$\bP$-almost every $\omega \in \Omega$, 
	$(X_n)_{0\le n<N}$ satisfies \eqref{eq:fcb} with $\mu=\nu_\omega$ and
	$$
	C_* = C (L+1) \Vert h_\omega \Vert_{BV},  \quad R(n) = \theta^n, 
	$$
	where $\theta \in (0,1)$ and $C > 0$ are constants depending only on 
	$\vartheta$ and the random dynamical system $(T_\omega)_{\omega \in \Omega}$.
\end{lem}

\begin{proof}
	Given a separately Lipschitz continuous function $F : [-L, L]^n \to \bR$ and index sets 
	$I, I_1, I_2$ as in \eqref{eq:fcb}, the function 
	$\tilde{F}(x_1, \ldots, x_n) = F(f_{i_1, \omega }(x_1), \ldots, f_{i_n, \omega}(x_n))$ satisfies 
	\[
	\text{ess sup}_{\omega \in \Omega} \|\tilde{F}\|_{\vartheta} \leq (L+1)\|F\|_{\Lip}.
	\]
	By rewriting $\int F \, d\nu_{I} - \int F \, d(\nu_{I_1} \otimes \nu_{I_2})$ as in \eqref{eq:rewrite_fcb}, we find
	that the result is a direct consequence of \cite[Proposition 3.5]{LS20}.
\end{proof}

Now, given a measurable bounded function $f : M \to \bR$, we set 
\[
f_{n, \omega} = f - \nu_\omega(f \circ T_\omega^n), \quad X_n = f_{n, \omega} \circ T_\omega^n.
\]
We consider $(X_n)_{n\ge 0}$ as a random process on $(M, \cF, \nu_\omega)$
and define the quantities $W$, $B$, and $Z$ as in Section~\ref{sec:main_thm_statement}.

\begin{thm} Suppose that $f : M \to \bR$ satisfies $\Vert f \Vert_{\vartheta}  \le L/2$ for 
	some $\vartheta \in (0,1]$,
	and that the function $\tilde{f}(\omega, x) = \tilde{f}_\omega(x) := f(x) - \nu_\omega(f)$ is 
	not a coboundary\footnote{I.e., $\tilde{f}$ cannot be written as $v  - v \circ \varphi$ for any $v \in L^2(\Omega \times M, \nu)$.}. Then, there exists 
	$N_0 = N_0(\omega) \ge 1$ such that 
	$\tilde{B} := \inf_{N \ge N_0} N^{-1} B_N > 0$ for $\bP$-a.e.~$\omega \in \Omega$. 
	Moreover, for any $g \in \sM_0(C_0)$ and $N \ge N_0$, 
	the following estimate holds for $\bP$-a.e.~$\omega \in \Omega$.
	\begin{align*}
		&| \nu_\omega [ g(W_N) ] - \bfE[g(Z)] |  \le 
		C C_\#  \Vert g \Vert_{\sM}  \tilde{B}^{-3/2} N^{-1/2}.
	\end{align*}
	Here, $C > 0$ is a constant depending only on $\vartheta$ and the random dynamical system $(T_\omega)_{\omega \in \Omega}$, and 
	$$
	C_\# =  (L + 1)^{ 9 / 2 } ( \Vert h_\omega \Vert_{BV} + 1  )^{3/2}   ( 
	1 +   C_0 ).
	$$
\end{thm}

\begin{proof} It follows from \cite{DFGV18} that the non-coboundary condition implies that there exists a nonrandom $\sigma^2 > 0$ such that, for 
	$\bP$-a.e. $\omega \in \Omega$, 
	$$
	\lim_{N\to \infty} N^{-1} \nu_\omega  \biggl[ \biggl( \sum_{n=0}^{N-1}  X_n  \biggr)^2 \biggr] = \sigma^2.
	$$
	Therefore, $\tilde{B} = \inf_{N \ge N_0} N^{-1} B_N > 0$  for a sufficiently large $N_0$, and the 
	result follows by applying Theorem \ref{thm:main} together with Lemma \ref{lem:fcb_LY}.
\end{proof}

\subsection{Application 4: Dispersing billiards}  Let $\bT^2 = \bR^2 / \bZ^2$ be the two-torus equipped with finitely many scatterers
$S_i$, $i = 1,\ldots, I$, which are pairwise disjoint closed convex sets such that the boundary $\partial S_i$ of each 
$S_i$ is a $C^3$ curve with strictly positive curvature. The billiard flow is defined by the motion of a point 
particle moving linearly at unit speed on the billiard table $\cQ = \bT^2 \setminus \cup_{i=1}^I S_i$, undergoing 
specular reflections at collisions with the boundaries $\partial S_i$ of 
the scatterers. We assume the finite horizon condition, which ensures a finite uniform upper bound on the time between consecutive collisions in $\cQ$.

The billiard flow induces a discrete-time billiard map $T : M \to M$, also known as the collision map, which 
keeps track of the collisions only. We adopt standard coordinates at collisions, $x = (r, \varphi)$, where 
$r$ represents the position of the billiard particle on the boundary $\cup_{i=1}^I \partial S_i$, 
parametrized by arc length, and $\varphi$ is the angle between the particle's post-collision velocity vector and the normal pointing into the domain
$\cQ$. 
The two-dimensional phase space $M$ is
a disjoint union of cylinders, $M = \cup_{i=1}^I \partial S_i \times [-\pi/2, \pi /2]$, and for a pair $x = (r, \varphi) \in M$, 
$T(x) = (r', \varphi')$ is defined as the corresponding pair after the next collision. The map $T$ is invertible 
and preserves 
a smooth probability measure $\mu$ on $M$, namely $d \mu = C^{-1}_\mu \cdot \cos r \, dr \, d \varphi$ where 
$C_\mu = \int_M  \cos r \, dr \, d \varphi$ is the normalizing factor \cite[Section 2.12]{CM06}. Given a bounded measurable 
function $f : M \to \bR$, we set $X_n = f \circ T^n$ and consider $(X_n)$ as a random process on 
$(M, \cF, \mu)$, where the $\sigma$-algebra $\cF$ consists of all Borel subsets of $M$.

To describe our application, we recall a few standard constructions from the theory of dispersing billiards, 
with full details provided in \cite{CM06}. Each cylinder $M_i =  \partial S_i \times [-\pi/2, \pi /2]$ is divided
into countably many connected components 
$M_{i,k} = \partial S_i \times \{ b_k < \varphi < b_{k+1} \}$, $k \in \bZ$, 
called homogeneity strips \cite[Section 5.4]{CM06}, where the 
numbers $b_k$ 
satisfy $b_{k + 1} - b_k = O(k^{-3})$. These strips $M_{i,k}$
accumulate near the boundary 
$\partial M_i = \partial S_i \times \{ \cos(\varphi) = 0 \}$ corresponding to tangential collisions. 
For $x, y \in M$, the future separation time 
$s_+(x,y)$ is defined as the smallest integer $n \ge 0$ such that $T^n(x)$ and $T^n(y)$ lie in different 
components. The past separation time is defined identically, using the inverse $T^{-1}$ instead of $T$.
A local stable manifold $W^s(x)$ of a point $x \in M$
is a maximal $C^2$ curve $W^s(x)$ such that for each $n \ge 0$ there exist 
$i,k$ with $T^n W^s(x) \subset M_{i,k}$. Almost every $x \in M$ has a nontrivial local 
stable manifold, and the (uncountable) family of all local stable manifolds forms a measurable partition of $M$. Local unstable manifolds are defined similarly, replacing $T$ with $T^{-1}$.

We use the definition of dynamical Hölder continuity from \cite{S10}, stated below, 
which is a variant of a corresponding definition in \cite{C06}.
The fact that this property is dynamically closed 
in the sense of \cite[Lemma 4]{S10} is used in the proof of the functional correlation bound \cite[Theorem 2.3]{LS17}, upon which our application is based.

\begin{defn} 
	A function $f: M \to \bR$ is said to be dynamically Hölder continuous on local unstable manifolds with rate $\vartheta_f \in (0,1)$ and constant $c_f \geq 0$ if it satisfies 
	$$
	|f(x) - f(y)| \le c_f \vartheta_f^{s_+(x,y)}
	$$
	whenever $x$ and $y$ belong to the same local unstable manifold. Define $\cH_+(c_f, \vartheta_f)$ as the family of all such functions. 
	Similarly, a function $f: M \to \bR$ is said to be dynamically Hölder continuous on local stable manifolds if it satisfies 
	$$
	|f(x) - f(y)| \le c_f \vartheta_f^{s_-(x,y)}
	$$
	whenever $x$ and $y$ belong to the same local stable manifold. Define $\cH_-(c_f, \vartheta_f)$ as the family of all such functions.
\end{defn}

\begin{remark}
If $f:X\to\bR$ is H\"older continuous with exponent $\alpha\in(0,1)$ and constant $[f]_\alpha$, then $f \in\cH_-(c,\vartheta)\cap \cH_+(c,\vartheta)$, where $c = [f]_\alpha$, and $\vartheta = \vartheta(\alpha)$ is determined by~$\alpha$ and system constants.
\end{remark}

\begin{lem}\label{lem:fcb_billiards}
Suppose that $f \in \cH_{-}(c_f, \vartheta_f) \cap \cH_{+}(c_f, \vartheta_f)$. 
Then, for all $N \ge 1$, $(X_n)$ satisfies \eqref{eq:fcb} with 
$$
C_* = C \biggl(   \frac{c_f}{1 - \vartheta_f}  + 1   \biggr), \quad R(n) = \theta^{n},
$$
where $\theta = \max\{  \vartheta_f, \theta_0 \}^{1/4}$, and $C > 0$, $\theta_0 \in (0,1)$ 
are system constants.
\end{lem}

\begin{proof} 
Let $I, I_1, I_2$, and $F$ be as in Definition \ref{defn:fcb}. As in the proof of 
Lemma \ref{lem:fcb_lsv}, define the function 
$$
\tilde{F}(x_1, \ldots, x_n) = F( f(x_1), \ldots, f(x_n) ).
$$
For each $x \in M^n$ and each $1 \le i \le n$, the function $\varphi_i : M \to \bR$, 
$\varphi_i(y) = F( x(y / i) )$, satisfies 
$\varphi_i \in \cH_{-}([F]_1 c_f, \vartheta_f) \cap \cH_+([F]_1 c_f, \vartheta_f)$. Hence, the result follows as a direct consequence of \cite[Theorem 2.4]{LS17}.
\end{proof}

\begin{thm}\label{thm:billiards} Suppose that $f \in \cH_{-}(c_f, \vartheta_f) \cap \cH_{+}(c_f, \vartheta_f)$ satisfies $\mu(f) = 0$. If $f$ is not 
	a coboundary, then $\tilde{B} := \inf_{N \ge 1} N^{-1} B_N > 0$. 
	Moreover,
	for any $N \ge 1$ and $g \in \sM_0(C_0)$,
	\begin{align*}
		&| \mu [ g(W_N) ] - \bfE[g(Z)] | 
		\le 
		C C_\theta C_\#  \Vert g \Vert_{\sM} \tilde{B}^{-3/2} N^{-1/2},
	\end{align*}
	where $C$ is a system constant, $C_\theta$ is a constant depending only on $\theta$ in Lemma \ref{lem:fcb_billiards}, and 
	$$
	C_\# =  \biggl(   \frac{c_f}{1 - \vartheta_f}  + 1   \biggr)^{3/2}( 1  + C_0  ) ( \Vert f \Vert_\infty + 1 )^3.
	$$
\end{thm}

\begin{proof} Since $\sum_{n=1}^\infty n |\mu(f \circ T^n \cdot f )|$ is summable, it is standard that 
the limit variance $\Sigma$ defined as in \eqref{eq:Sigma} is positive if $f$ is not a coboundary. Hence, 
$\tilde{B} > 0$ follows as in the proof of Theorem \ref{thm:pikovsky}.
The desired upper bound then follows directly from Theorem
\ref{thm:main} and Lemma \ref{lem:fcb_billiards}.
\end{proof}

\section{Review of Stein's method for Brownian approximations}\label{sec:stein}

In this section, we briefly review Stein's method in the context of 
diffusion approximations. We refer the reader to 
\cite{B90, K20, KDV17, BRZ23} for more details. Since the results of our 
paper concern approximation by the standard univariate Brownian motion, we shall review Stein's method 
in this case.  

Throughout this section, $C$ denotes an absolute constant,
the value of which may vary between different expressions.

\subsection{Stein's method for diffusion approximations} 
Recall the spaces of test functions 
$\sL$ and $\sM$ from Section \ref{sec:spaces}. 

A consequence of Proposition \ref{prop:M_norm} and Taylor's theorem is that, for each $g \in \sM$ and each $w,h \in D$,
\begin{align}\label{eq:taylor}
	\begin{split}
		&\Vert g'(w) \Vert \le K_g (1 + \Vert w \Vert^2_\infty ), \quad  \Vert g''(w) \Vert \le  
		K_g 
		 (1 + \Vert w \Vert_\infty ),  \\
		&| g(w + h) -g(w) - g'(w)[h]  - \tfrac12 g^{(2)}(w)[h^{(2)}] |  \le  K_g  \Vert h \Vert_\infty^3,
	\end{split}
\end{align}
where $K_g = C \Vert g \Vert_{\sM}$.

Let $Z$ be a standard Brownian motion on $[0,1]$. In 
\cite{B90}, Barbour derived a Stein equation for approximation by $Z$ as 
$\cA f = g - \bfE(g(Z))$
where $g \in \sM$ and $\cA$ is the generator of 
a Markov process whose stationary law is the Wiener measure. The construction of Brownian motion 
by Schauder functions was used to define the appropriate process. Below, we recall the construction of $\cA$ from \cite{B90}, along with some of its properties, which will be used in the sequel.

Let 
\begin{align}\label{eq:process_x_hat}
	\{ (\hat{X}_k(u), \: u \ge 0) \: : \: k =0,1,2,\ldots \}
\end{align}
be an i.i.d. collection of Ornstein--Uhlenbeck processes on $[0, \infty)$
with stationary law $N(0,1)$, i.e. independent processes such that each $\hat{X}_k$ weakly solves the 
stochastic differential equation
$$
d x_t = - x_t \, dt + \sqrt{2} \, d Z_t, \quad x_0 \sim N(0,1), \quad t \ge 0.
$$
Define 
$$
\hat{W}(t, u) = \sum_{k=0}^\infty \hat{X}_k(u) S_k(t), \quad 0 \le t \le 1, \quad u \ge 0,
$$
where the Schauder functions $S_k$ are given by
\begin{align*}
	S_0(t) = t, \quad S_k(t) = \int_0^t H_k(u) \, d u, \quad k \ge 1,
\end{align*}
and, for $2^n \le k < 2^{n+1}$, 
\begin{align*}
	H_k(u) &= 2^{n/2} \biggl\{ 
	\bfI[   2^{-n}k - 1 \le u \le 2^{-n} ( k + 1/2  ) - 1  ] \\
	&- 
	\bfI[   2^{-n} ( k + 1/2 ) - 1  \le u \le 2^{-n} ( k + 1  ) - 1  ]
	\biggr\}.
\end{align*}
The function $(t,u) \mapsto \hat{W}(t,u)$ is almost surely continuous, and for each $u$, $\hat{W}( \cdot, u)$ is distributed 
according to the Wiener measure. The following two equations correspond to \cite[equation (2.9)]{B90} and \cite[equation (2.11)]{B90}, respectively. We provide a sketch of the proof following \cite{B90}.

\begin{lem}\label{lem:gen}
The infinitesimal generator $\cA$ of the process $(\hat{W}(\cdot, u))_{u \ge 0}$ acts on any $f \in \sM$ in the following way:	
\begin{align}\label{eq:generator}
\cA f ( w ) = -  f'(w) [w] + \bfE  f''(w) [  Z^{(2)} ] = -  f'(w) [w]  + \sum_{k = 0}^\infty  f''(w) [ S_k^{(2)} ].
\end{align}
Here, $f''(w)[z^{(2)}]$ denotes $f''(w)[z,z]$ for $z\in D$.
\end{lem}

\begin{proof}[Sketch of proof] The semigroup $(T_u)_{u \ge 0}$ of $(\hat{W}(\cdot, u))_{u \ge 0}$ acting on $\sL$ is given by the formula
\begin{align}\label{eq:semigroup}
(T_u f )(w) = \bfE[  f(  w e^{-u}  + \sigma(u) Z ) ], 
\end{align}
where $\sigma^2(u) = 1 - e^{-2u}$. For $u, v \ge 0$, we have (see \cite[equation (2.3)]{B90} or 
\cite[Lemma 5.4]{K20})
$$
\hat{W}(\cdot, u+ v) - e^{-v} \hat{W}(\cdot, u) \stackrel{ \mathscr{D} }{ = }  \sigma(v) Z( \cdot ).
$$ 
Using the bilinearity of $f^{(k)}$ along with $1 - e^{-u} = u + O(u^2)$ and \eqref{eq:taylor}, we find that
\begin{align*}
	&(T_u f )(w)  - f(w) + u f'(w)[w] - u \bfE \{  f^{''} (w)[Z^{(2)}] \} \\
	 &= R
	+O(u^2 K_f (  1 + \Vert w \Vert_\infty^2 )  + u^2 K_f (  1 + \Vert w \Vert_\infty )  + u^2 K_f (  1 + \Vert w \Vert_\infty^3 )  ),
\end{align*}
where $K_f = C \Vert f \Vert_{\sM}$,  and 
\begin{align*}
R &= (T_u f)(w) - f(w) \\
&+ \bfE\biggl\{   f'(w) [  w( 1 - e^{-u} ) ]   - \frac12 \sigma^2(u) f''(w) [Z^{(2)}] 
- (1 - e^{-u})^2 f''(w)[ w^{(2)}] \biggr\}.
\end{align*}

Next, from the representation 
\begin{align}\label{eq:repr_bm}
Z(t) = \sum_{k = 0}^\infty Z_k S_k(t),
\end{align}
where $Z_k \sim N(0,1)$ are i.i.d., 
we see that  $\bfE\{  f'(w) [ Z ] \} = 0 = \bfE\{  f''(w) [ w, Z ] \} $.
Therefore, by \eqref{eq:taylor} and the inequality $1 - e^{-u} \le u$ for $u \ge 0$, it follows that for any $u \in [0,1]$,
\begin{align*}
|R| =
 \biggl|	(T_u f)(w) &- f(w) - \bfE\biggl\{   f'(w) [  \sigma(u) Z - w( 1 - e^{-u} ) ]   \\
	&- \frac12 f''(w) [  (  \sigma(u) Z - w(1 - e^{-u}) )^{(2)}  ]
	\biggr\}  \biggr|\\
 &\le C \Vert f \Vert_{\sM} \bfE[ \Vert    \sigma(u) Z - w(1 - e^{-u})  \Vert_\infty^3  ] \\
 &\le C \Vert f \Vert_{\sM}  \bfE[  ( \sqrt{2u} |Z| + u \Vert w \Vert_\infty )^3 ] \\
 &\le C \Vert f \Vert_{\sM}  (1 + \Vert w \Vert_\infty^3 )  u^{3/2}.
\end{align*}
This yields the first equality in \eqref{eq:generator}:
$$
\cA f( w )  = \lim_{u \downarrow 0} \frac{(T_u f - f)(w)}{u} = - f'(w) [w] + \bfE  f''(w) [  Z^{(2)} ].
$$
The second equality is a straightforward consequence of \eqref{eq:repr_bm}.
\end{proof}

For any $g \in \sM$ with $\bfE[g(Z)] = 0$, the Stein 
equation is defined as 
\begin{align}\label{eq:se}
	\cA f = g, \tag{SE}
\end{align}
where $\cA$ is as in \eqref{eq:generator}.

Let 
\begin{align}\label{eq:z_n}
	Z_n(t) = n^{-1/2} \sum_{k=1}^{  \round{nt}  } Z_k,
\end{align}
where $Z_k \sim N(0,1)$ are independent random variables. Further, let 
$$
\hat{W}_n(t, u) = \sum_{k=1}^n \hat{X}_k(u)   J_{ k / n }(t) , 
$$
where we recall that $\hat{X}_k$ are defined in \eqref{eq:process_x_hat}
and $J_{k/n}$ are defined in \eqref{eq:Jt}.
Then, the stationary distribution of $( \hat{W}_n( \cdot , u)  )_{u \ge 0}$ is the distribution of $Z_n$, and 
similar to Lemma \ref{lem:gen} it can be shown (see \cite[Proposition 5.1]{K20}) that the infinitesimal generator $\cA_n$ 
of $( \hat{W}_n( \cdot , u)  )_{u \ge 0}$ acts on any $f \in \sM$ as follows:
$$
\cA_n f ( w ) = -  f'(w) [w] + \bfE  f''(w) [  Z_n^{(2)} ].
$$

Let $\sM'$ consist of all $f \in \sM$ such that 
\begin{align}\label{eq:m'_cond}
	\lim_{n \to \infty} 
		\int_0^1 f''(w)[ J_{  \round{nt} /n   }^{(2)} ] \, dt = \int_0^1 f''(w) [  J_t^{(2)} ] \, dt \quad \forall w \in D.
\end{align}
The following equation corresponds to \cite[equation (2.18)]{B90}. See also 
\cite[Theorem 2]{BRZ23} for a corresponding result in the case of general multidimensional 
Gaussian processes. We include a proof for completeness.

\begin{lem}\label{lem:generator} For $f \in \sM'$, the generator
$\cA$ in \eqref{eq:generator} satisfies
	\begin{align}\label{eq:gen-2}
		( \cA f )(w) = - f'(w)[w] + \int_0^1  f''(w) [  J_t^{(2)} ] \, dt.
	\end{align}
\end{lem}

\begin{proof} Recalling \eqref{eq:repr_bm}, 
	it suffices to verify that 
	\begin{align*}
		\int_0^1  f''(w) [  J_t^{(2)} ] \, dt = \bfE f''(w)[Z^{(2)}] \quad \forall w \in D.
	\end{align*}
	Define 
	$$
	G_n(t) = Z(  \ell / n  ) \quad \forall t \in [ \ell / n, ( \ell + 1) / n  ],
	$$
	where $Z$ is a standard Brownian motion. Observe that $G_n$
	has the same disribution as $Z_n$ in \eqref{eq:z_n}. By the bilinearity of $f''(w)$ and Proposition \ref{prop:M_norm},
	\begin{align*}
		&| \bfE f''(w) [Z^{(2)} ] -   \bfE f''(w) [G^{(2)}_n ]  | \\
		&= | \bfE[   f''(w) [ Z - G_n , Z  ]   + f''(w) [G_n, Z - G_n] ]   | \\
		&\le C  \Vert f \Vert_{\sM} (1 + \Vert w \Vert_\infty) \bfE \biggl[ 
		\sup_{ t \in [0,1] } |  Z(t) - Z(  \round{nt}/n  )    |   \sup_{t \in [0,1]} |Z(t)|
		\biggr] \\
		&\le C  \Vert f \Vert_{\sM} (1 + \Vert w \Vert_\infty) \biggr\{ \bfE \biggl[ 
		\sup_{ t \in [0,1] } |  Z(t) - Z(  \round{nt}/n  )    |^2
		\biggr] \biggl\}^{1/2} \\
		&\le  C  \Vert f \Vert_{\sM} (1 + \Vert w \Vert_\infty)  \log(n)^{1/2} n^{-1/2} \to 0, \quad n \to \infty,
	\end{align*}
	where the last inequality follows from \cite[Lemma 3]{FN10}. Using \eqref{eq:z_n}, we now obtain
	$$
	\bfE f''(w) [Z^{(2)} ] = \lim_{n \to \infty}   \bfE f''(w) [G^{(2)}_n ] = \lim_{n \to \infty} \int_0^1 f''(w) [  J_{ \round{nt} /n  }^{(2)} ] \, dt.
	$$
	Since $f \in \sM'$,
	$$
	\lim_{n \to \infty} \int_0^1 f''(w) [  J_{ \round{nt} /n  }^{(2)} ] \, dt = \int_0^1 f''(w)[J_t^{(2)}] \, dt.
	$$
\end{proof}

\subsection{On the existence and properties of the solution to \eqref{eq:se}}
In \cite[Lemma 3.1]{KDV17}, an explicit counterexample is given to show that 
the semigroup $(T_u)_{u \ge 0}$ in \eqref{eq:semigroup}  is not strongly continuous on $(\sL, \Vert \cdot \Vert_\sL)$. Nevertheless, it is established that $\phi(g) := - \int_0^\infty T_u g \, du$ solves \eqref{eq:se} for any $g \in \sM$:

\begin{lem}[Lemma 4.1 and Proposition 4.4 in \cite{KDV17}]\label{lem:stein_sol}\phantom{text}
\begin{itemize}
	\item[(1)] For any $g \in \sM$ such that $\bfE[g(Z)] = 0$, $f = \phi(g) = - \int_0^\infty T_u g \, du$ exists 
	and satisfies $\phi(g) \in \sM$. In particular, $\phi(g)$ is in the domain of $\cA$. Moreover, for $k = 1,2$,
	\begin{align}\label{eq:Dk_fml}
		 ( \phi(g) )^{(k)} (w) = - \bfE \int_0^\infty e^{-ku}  g^{(k)}(  w e^{-u}  + \sigma(u) Z  ) \, du.
	\end{align}
	\item[(2)] For any $g \in \sM$ with $\bfE[g(Z)] = 0$, $f = \phi(g)$ solves \eqref{eq:se}: $\cA f = g$.
\end{itemize}

\end{lem}

From \eqref{eq:Dk_fml} it follows that 
\begin{align}\label{eq:phi_bound}
\Vert \phi( g ) \Vert_{\sM} \le C \Vert g  \Vert_{\sM}.
\end{align}

The following result, which is taken from \cite{B90}, is a straightforward consequence of \eqref{eq:Dk_fml}. We include a proof for completeness, as it was omitted from \cite{B90}.

\begin{lem}\label{lem:smooth} Recall from \eqref{eq:smooth} that $\sM_0$ is 
defined as the class of all $g \in \sM$ that satisfy the smoothness condition 
	\begin{align*}
		\sup_{w \in D }  | g''(w)[J_r, J_s - J_t] |  \le C_0 \Vert g \Vert_\sM |t-s|^{1/2} \quad \forall r,s,t \in [0,1].
	\end{align*}
	If $g \in \sM_0$, then $\phi(g) \in \sM'$.
\end{lem}

\begin{proof} Let $f = \phi(g)$. By \eqref{eq:Dk_fml},
	\begin{align*}
		&\int_0^1 f''(w) [  J_t^{(2)} ] -  f''(w)[ J_{  \round{nt} /n   }^{(2)} ] \, dt  \\
		&=-  \int_0^1   \bfE \int_0^\infty e^{-2u}  \biggl\{  
		g''(  w e^{-u}  + \sigma(u) Z  )[  J_t^{(2)} ]  - g''(  w e^{-u}  + \sigma(u) Z  ) [J_{  \round{nt} /n   }^{(2)} ]   \biggr\}  \, du  
		\, dt 
	\end{align*}
	Since $f''(w)$ is symmetric, we can apply \eqref{eq:smooth} twice to obtain
	\begin{align*}
		&\biggl| 
		\int_0^1 f''(w) [  J_t^{(2)} ] -  f''(w)[ J_{  \round{nt} /n   }^{(2)} ] \, dt 
		\biggr| \\
		&\le \int_0^1 \int_0^\infty e^{-2u}  \cdot 2 C_0 \Vert g \Vert_{\sM} | t - \round{nt}/n  |^{1/2} \, du \, dt \\
		&\le  C C_0  \Vert g \Vert_\sM n^{-1/2}.
	\end{align*}
\end{proof}

\subsection{Conclusion of abstract Barbour's method}
Suppose that $g \in \sM_0$. Then, $h := g - \bfE[ g(Z)] \in \sM_0$ satisfies 
$\bfE[ h (Z)] = 0$. By Lemmas \ref{lem:stein_sol} and \ref{lem:smooth},
for any $D$-valued random variable $X$ on $(M,\mathcal F,\mu)$,
$$
\mu[ g(X) ] - \bfE[  g(Z) ] = \mu[ h(X)] = \mu[  \cA f (X)  ],
$$
where $f = \phi(h) \in \sM'$. Therefore, by Lemma \ref{lem:generator},
\begin{align}\label{eq:1216b}
	\mu[ g(X) ] - \bfE[  g(Z) ] = \mu \biggl\{  
	- f'(X)[X] + \int_0^1  f''(X) [  J_t^{(2)} ] \, dt
	\biggr\}.
\end{align}
Notice that, in this step, we do not use any concrete property for $X$.

\section{Proof of Theorem \ref{thm:main}}\label{sec:main_proof}

Recall the definition of $W_N$ from \eqref{eq:def_W}.
In this section, we write $W = W_N$, omitting the subscript for brevity.
To prove Theorem \ref{thm:main},
we will estimate the right-hand side of \eqref{eq:1216b} with $X=W$.
We begin by formulating a preliminary version of Theorem \ref{thm:main} that 
reveals the exact class of test functions $F$ for which the bound \eqref{eq:fcb} is required.

\subsection{A preliminary result} 

Recall the definition of $\theta_n(t)$ from \eqref{eq:def_theta}.
For $0 \le n,m < N$, define 
the following auxiliary random elements:
\begin{align}\label{eq:Snm}
	S_{n,m}(t) &= \sum_{ \substack{ 0 \le i < N \\  |i - n| > m  } }   \theta_i(t)  X_i \quad 
	\text{and} \quad T_{n,m}(t) = \sum_{ \substack{ 0 \le i < N \\  |i - n| = m  } }   \theta_i(t)  X_i.
\end{align}
Let $Z$ be a standard Brownian motion independent from $(X_n)$. Given $g \in \sM$, for 
any $s,t,u \in [0,1]$ and $0 \le i,n < N$ define 
\begin{align}
	\begin{split}
		\label{eq:Hg}
		&H_g(w_1, w_2; s,t,u, i, n) \\
		&= \bfE \biggl\{  g''(  s  B^{-1/2} w_1 + t B^{-1/2} w_2 + u Z )[\theta_i, \theta_n] 
		-  g''(  s B^{-1/2} w_1 + u Z )[\theta_i, \theta_n] \biggr\},
	\end{split}
\end{align}
where $\bfE$ denotes expectation with respect to $Z$.

For a real-valued random variable $X$ defined on $(M, \cF, \mu)$, let  
$\overline{X} = X - \mu(X)$.

\begin{thm}\label{thm:prelim} Let $N \ge 1$ and $g \in \sM_0(C_0)$. Suppose there exist a function 
	$\rho : \bZ_+ \to \bR_+$ and constants $C_i$, $1 \le i \le 3$, such that the following conditions 
	hold for all $0 \le n,m < N$ and all $s,t,u \in [0,1]$:
	\begin{itemize}
		\item[(A1)] $|\mu(X_n X_m)| \le C_1 \rho(|n - m|)$. \smallskip 
		\item[(A2)] Whenever $|i -n | = m$, 
		$0 \le i < N$, and $m \le k < N$, 
		$$\biggl|\mu \biggl\{  X_n X_i   H_g(  S_{n,k}, T_{n,k} ; s,t,u,i, n)  \biggr\} \biggr| \le 
		C_2 B^{-1/2} \Vert g \Vert_{\sM } \rho(m).$$ \smallskip 
		\item[(A3)] Whenever $|i -n | = m$, 
		$0 \le i < N$, and $2m + 1 \le k < N$
		$$
		\biggl|   \mu \biggl\{  X_n X_i   \overline{H_g(  S_{n,k}, T_{n,k} ; s,t,u,i, n)}  \biggr\}   \biggr| \le 
		C_3 B^{-1/2} \Vert g \Vert_{\sM } \rho(k-m).
		$$
	\end{itemize}
	Then, 
	$$
	| \mu [ g(W_N) ] - \bfE[g(Z)] |  \le C C_N  \Vert g \Vert_{\sM}  B^{-3/2} N,
	$$
	where $C > 0$ is an absolute constant,
	$$
	C_{N} = L^3 + (C_1 L  + C_2 + C_3) \sum_{m=1}^{N-1} m \rho(m) + C_0 C_1^{3/2} C_\rho^{1/2}\biggl[ C_\rho
	+   \sum_{j=1}^{N-1} j^{1/2} \rho(j) \biggr],
	$$
	and $C_\rho = \sum_{j=0}^{N-1} \rho(j)$.
\end{thm}

The proof of Theorem \ref{thm:prelim} is provided in the following section.
In the proof of Proposition \ref{prop:0118}, we verify that the conditions in 
Theorem \ref{thm:prelim} can be described as instances of \eqref{eq:fcb2} with at most three blocks of $X_n$'s
 (namely, \eqref{eq:fcb2} with $K = 2$ and $3$ suffices for (A1)-(A3)). 
Using Proposition \ref{prop:fcb1-k}, we then conclude that (A1)-(A3) follow from \eqref{eq:fcb}. Compare also (A1)-(A3) with the corresponding conditions (B1)-(B3) in \cite[Theorem 2.3]{LS20} for normal approximations for chaotic systems. 

\subsection{Proof of Theorem \ref{thm:prelim}}\label{sec:proof_of_prelim_thm}

\subsubsection{Decomposition of $\mu[  \cA f (W)  ]$}\label{s:52}
By \eqref{eq:1216b}, for each $g \in \sM_0$, 
$$
\mu[ g(W) ] - \bfE[  g(Z) ] = \mu[  \cA f (W)  ]= \mu \biggl\{  
- f'(W)[W] + \int_0^1  f''(W) [  J_t^{(2)} ] \, dt
\biggr\},
$$
where $f = \phi(h) \in \sM'$ with $h= g - \bfE[ g(Z)] $. Recall that 
$\sigma_{n,m} = \mu(X_n X_m)$. We then have the decomposition 
\begin{align}\label{eq:decomp}
\mu[  g(W) ]  - \bfE[ g(Z) ] = I + II
\end{align}
where 
\begin{align*}
	I = \mu \biggl\{  - f'( W )[W] + B^{-1} \sum_{n=0}^{N-1}  \sum_{m=0}^{N-1}    \sigma_{n,m}   f''(W) [ \theta_n, \theta_m ]  \biggr\}, 
\end{align*}
and 
\begin{align*}
	II = \mu \biggl\{  -
	B^{-1} \sum_{n=0}^{N-1}  \sum_{m=0}^{N-1}    \sigma_{n,m}   f''(W) [ \theta_n, \theta_m ] + 
	\int_0^1  f''(W) [  J_t^{(2)} ] \, dt
	\biggr\}.
\end{align*}
Below, we estimate $I$ and $II$ to establish the bound in Theorem \ref{thm:prelim}.
Term $I$ corresponds to estimating how close the distribution of $W$ is to that of 
$Y = B^{-1/2} \sum_{n=0}^{  N - 1  } \tilde{Z}_n \theta_n$, where 
$\tilde{Z}_n$ are Gaussian random variables, independent of $(X_n)$, having the same covariance 
structure as $(X_n)$:
$$
\bfE(   \tilde{Z}_i \tilde{Z}_j ) = \mu(X_i X_j) \quad \forall i, j \ge 0.
$$
This is made precise by \cite[Proposition 5.1]{K20}.

\subsubsection{Estimate on $II$} In the rest of the proof we write $A \lesssim B$ if there exists an absolute constant 
$C > 0$ such that $A \le C B$. Below we estimate $II$ using (A1). Note that (A2) and (A3) are not required for this estimate.

\begin{prop}\label{prop:estim_ii}
	Assuming (A1),
	\begin{align}\label{eq:estim_II-2}
		|II| \lesssim \Vert g \Vert_{\sM} B^{-3/2} N C_0 C_1^{3/2}C_\rho^{1/2}\biggl[ C_\rho
		+   \sum_{j=1}^{N-1} j^{1/2} \rho(j) \biggr],
	\end{align}
	where $C_\rho = \sum_{k = 0}^{N-1}  \rho(k)$.
\end{prop}

\begin{proof} Recall that $B = \sum_{n=0}^{N-1} \beta_n = \sum_{n=0}^{N-1}   \sum_{m = 0}^{N-1} \sigma_{n,m} $ and	
	decompose $II = - R_1 - R_2$, where 
	\begin{align*}
		R_1 = \mu \biggl\{ 
		B^{-1} \sum_{n=0}^{N-1}  \sum_{m=0}^{N-1}    \sigma_{n,m}   \biggl(   f''(W) [ \theta_n, \theta_m ] -  f''(W) [ \theta_n, \theta_n ] \biggr) 
		\biggr\} 
	\end{align*}	
	and 
	\begin{align*}
		R_2 =  \mu \biggl\{  
		B^{-1} \sum_{n=0}^{N-1}  \beta_n  f''(W) [ \theta_n, \theta_n ] - 
		\int_0^1  f''(W) [  J_t^{(2)} ] \, dt \biggr\}.
	\end{align*}
	We use \eqref{eq:Dk_fml} to compute $f''$ and then apply \eqref{eq:smooth} to obtain
	\begin{align*}
		|R_1| &=  \biggl| \mu \biggl\{ 
		B^{-1} \sum_{n=0}^{N-1}  \sum_{m=0}^{N-1}    \sigma_{n,m}  
		\bfE \int_0^\infty e^{-2u}  g''(  W e^{-u}  + \sigma(u) Z  )[ \theta_n, \theta_m    ] \\
		&- e^{-2u}  g''(  W e^{-u}  + \sigma(u) Z  )[ \theta_n, \theta_n   ]     \, du
		\biggr\} \biggr| \\
		&\le B^{-1} \sum_{n=0}^{N-1}  \sum_{m=0}^{N-1}   | \sigma_{n,m} |  \int_0^\infty e^{-2u} \, d u \cdot 
		C_0 \Vert g \Vert_\sM | ( B_n -  B_m ) / B  |^{1/2} \\ 
		&\le C_0 \Vert g \Vert_\sM B^{-3/2} \sum_{n=0}^{N-1}  \sum_{m=0}^{N-1}   |  \sigma_{n,m} |   |  B_n -  B_m  |^{1/2}.
	\end{align*}
	Using $\beta_n = B_{n+1} - B_n$, we see that 
	\begin{align*}
		R_2 &= \mu \biggl\{  \sum_{n=0}^{N-1}  \int_{  B_n / B }^{ B_{n+1} / B  }  f''(W) [ J^{(2)}_{ B_n /  B } ]  \, dt - 
		\sum_{n=0}^{N-1}  \int_{ B_n / B }^{ B_{n+1} / B  }  f''(W) [  J_t^{(2)} ] \, dt \biggr\} \\ 
		&= \mu \biggl\{  \sum_{n=0}^{N-1}  \int_{  B_n / B }^{  B_{n+1} / B  }  
		\bfE \int_0^\infty e^{-2u}  g''(  W e^{-u}  + \sigma(u) Z  )[ J^{(2)}_{ B_n /  B }  ] \\
		&- e^{-2u}  g''(  W e^{-u}  + \sigma(u) Z  )[ J_{ t }, J_{ t }    ]     \, du
		\, dt \biggr\}.
	\end{align*}
	Hence, by applying \eqref{eq:smooth} twice together with the symmetricity of $g''(w)$, we arrive at the estimate
	\begin{align*}
		|R_2| \lesssim \sum_{n=0}^{N-1}  C_0 \Vert g \Vert_\sM |  B_n / B - B_{n+1} / B  |^{3/2} 
		\le C_0 \Vert g \Vert_\sM    \sum_{n=0}^{N-1}  | \beta_n / B |^{3/2}.
	\end{align*}
	Finally, using (A1), we find that 
\begin{align*}
	&\sum_{n=0}^{N-1}  | \beta_n / B |^{3/2} = B^{-3/2} \sum_{n=0}^{N-1} \biggl|  \sum_{m = 0}^{N-1} \sigma_{n,m} \biggr|^{3/2}
	 \le         C_1^{3/2} B^{-3/2} \sum_{n=0}^{N-1} \biggl(  \sum_{m = 0}^{N-1} \rho(\vert n-m\vert) \biggr)^{3/2} \\
	&\lesssim 
	C_1^{3/2} B^{-3/2} \sum_{n=0}^{N-1} \biggl(  \sum_{k = 0}^{N-1}  \rho(k)   \biggr)^{3/2}
	=
	C_1^{3/2} B^{-3/2} N \biggl(  \sum_{k = 0}^{N-1}  \rho(k)   \biggr)^{3/2}
\end{align*}
and
\begin{align*}
	&\sum_{n=0}^{N-1}  \sum_{m=0}^{N-1}   |  \sigma_{n,m} |   |  B_n -  B_m  |^{1/2}
	\lesssim C_1  \sum_{n=0}^{N-1}  \sum_{m=0}^{N-1}   \rho(|n-m|)   |  B_n -  B_m  |^{1/2} \\ 
	&=  C_1     \sum_{n=0}^{N-1}  \sum_{m=0}^{N-1}  \rho( |n - m| )  \biggl|    \sum_{p=\min\{n,m\}}^{\max\{n,m\}-1}\sum_{q= 0}^{N-1}\sigma_{p,q} \biggr|^{1/2}  \\
	&\lesssim  C_1^{3/2}    \biggl(  \sum_{k = 0}^{N-1}  \rho(k)   \biggr)^{1/2}   \sum_{n=0}^{N-1}  \sum_{m=0}^{N-1} |n - m|^{1/2}  \rho( |n - m| )    \\
	&\lesssim C_1^{3/2} N \biggl(  \sum_{k = 0}^{N-1}  \rho(k)   \biggr)^{1/2}  \sum_{j=1}^{N-1} j^{1/2} \rho(j).
\end{align*}
\end{proof}

\subsubsection{Estimate on $I$} For brevity, denote
$$
\tilde{X}_n = B^{-1/2} X_n, \quad \text{and} \quad 
Y_n = \tilde{X}_n \theta_n.
$$
\begin{remark}
The specific form of $\theta_n$ plays no role in what follows. Indeed, all of the results in this section 
continue to hold if $\theta_n$ are replaced with arbitrary $\theta_n^* \in D$ such that 
$\Vert \theta_n^* \Vert_\infty \le 1$, provided that (A2) and (A3) hold with  $\theta_n^*$ in place of 
$\theta_n$.	
\end{remark}

Given integers $0 \le n < N$ and $m \in \bZ$, let $W_{n,m} = B^{-1/2} S_{n,m}$ and $Y_{n,m} =  B^{-1/2} T_{n,m}$ where $S_{n,m}$ and $T_{n,m}$ 
are defined as in \eqref{eq:Snm}.
Note that 
$$W_{n, N-1} = 0 \quad \text{and} \quad  W_{n,-1} = W.$$ 
We will estimate $I$ using (A1)-(A3) together with the following decomposition, which is a counterpart of 
\cite[Proposition 4]{S07}:

\begin{prop}\label{prop:decomp} For $s \in [0,1]$, $0 \le i, n < N$, and $m \in \bZ$, define 
	$$
	\xi_{i, n,m}(s) =  f''(   W_{n,m} + s Y_{n,m}     )[   \theta_i ,  \theta_n] \,   - f''(   W_{n,m}    )[   \theta_i ,  \theta_n]
	$$
	and
	$$
	\xi_{i, n,m} = \xi_{i, n,m}(1) =  f''(   W_{n,m-1}    )[   \theta_i ,  \theta_n] \,   - f''(   W_{n,m}  )[   \theta_i ,  \theta_n].
	$$
	Then, $I = -\sum_{i=1}^7 E_i$, 
	where $E_i = E_i(f)$ are given by
	\begin{align*}
		E_1 &= \sum_{n=0}^{N-1} \sum_{m=1}^{N-1}  \sum_{ \substack{|i - n| = m \\ 
				0 \le i < N } }  \int_0^1  \mu ( \tilde{X}_n 
		\tilde{X}_i   \xi_{i,n,m}(s)  )  \, ds  , \quad 
		E_2 = \sum_{n=0}^{N-1}      \int_0^1  \mu(   \tilde{X}_n^2  \xi_{n, n,0}(s) )  \, ds, \\
		E_3 &= \sum_{n=0}^{N-1}  \sum_{m=1}^{N-1}   \sum_{k=m + 1}^{2m}    \sum_{ \substack{|i - n| = m \\ 
				0 \le i < N } } \mu ( 
		\tilde{X}_n   \tilde{X}_i    \overline{\xi_{i, n,k}}   ), \quad 
		E_4 = \sum_{n=0}^{N-1}  \sum_{m=1}^{N-1}   \sum_{k= 2m + 1}^{N - 1}    \sum_{ \substack{|i - n| = m \\ 
				0 \le i < N } } \mu (
		\tilde{X}_n   \tilde{X}_i     \overline{\xi_{i, n,k}} ), \\ 
		E_5 &= \sum_{n=0}^{N-1}    \sum_{k= 1}^{N - 1}   \mu (
		\tilde{X}_n^2     \overline{\xi_{n, n,k}}  ), \quad 
		E_6 = - \sum_{n=0}^{N-1}  \sum_{m=1}^{N-1} \sum_{k=0}^m  \sum_{ \substack{|i - n| = m \\ 
				0 \le i < N } } \mu( \tilde{X}_n
		\tilde{X}_m  )   \mu(  \xi_{i,n,k}  ),   \\
		E_7 &= - \sum_{n=0}^{N-1}  \mu( \tilde{X}_n^2 ) \mu(  \xi_{n,n,0}  ).
	\end{align*}
	
\end{prop}

\begin{proof} Recall that 
\begin{align*}
	- I = \mu \biggl\{   f'( W )[W] - B^{-1} \sum_{n=0}^{N-1}  \sum_{m=0}^{N-1}    \sigma_{n,m}   f''(W) [ \theta_n, \theta_m ]  \biggr\}.
\end{align*}
By linearity, $\mu (    f' (W)[ W ] ) = \sum_{n=0}^{N-1}   \mu (  \tilde{X}_n    f'(W) [  \theta_n  ]   )$.
	Since $\mu( \tilde{X}_n) = 0$, we have 
	$$
	\mu( \tilde{X}_n f'( W_{n, N-1} )[ \theta_n  ]  ) = \mu( \tilde{X}_n f'( 0 )[ \theta_n  ]  ) = f'( 0 )[ \theta_n ]  \mu( \tilde{X}_n) = 0.
	$$
	Hence, we can express $\mu (  f' (W)[ W ]  )$ as the following telescopic sum:
	\begin{align}
		\mu (  f' (W)[ W ]  ) &=  \sum_{n=0}^{N-1}   \mu (  \tilde{X}_n    f'(W) [  \theta_n  ]   ) 
		= \sum_{n=0}^{N-1}   \mu \biggl\{  \tilde{X}_n  \biggl(    f'(W_{n,-1}) [  \theta_n ]   -  f'(  W_{n,N-1} ) [  \theta_n  ]    \biggr)   
		\biggr\}  \notag \\ 
		&=  \sum_{n=0}^{N-1} \sum_{m=0}^{N-1}   \mu \biggl\{  \tilde{X}_n  \biggl(    f'(W_{n, m-1}) [  \theta_n  ]   -   f'(  W_{n,m} ) [  \theta_n  ]    \biggr)   
		\biggr\}. \label{eq:sunk-1}
	\end{align}
	Since $f' \in C^2$, we have 
	\begin{align*}
		f'(x + h)[ \cdot ] - f'(x)[ \cdot ] =	\int_0^1 f''( x  + s h  )[h, \cdot ] \, ds, \quad x,h \in D,
	\end{align*}
	where the integral is defined as a Bochner integral. Using the relation $W_{n,m} + Y_{n,m} = W_{n,m-1}$, 
	we find that 
	\begin{align}\label{eq:int_dif}
		f'( W_{n,m-1}  ) [ \theta_n ] - f'(W_{n,m}) [ \theta_n  ] = \int_0^1 f''(   W_{n,m} + s Y_{n,m}    )[  Y_{n,m} , \theta_n] \,  ds.
	\end{align}
	Substituting \eqref{eq:int_dif} into \eqref{eq:sunk-1},
	\begin{align*}
		\mu (  f' (W)[ W ] ) 
		&= 
		\sum_{n=0}^{N-1} \sum_{m=0}^{N-1}   \mu \biggl\{  \tilde{X}_n   \int_0^1 f''(   W_{n,m} + s Y_{n,m}     )[  Y_{n,m},  \theta_n] \,  ds
		\biggr\} \\
		&=  \sum_{n=0}^{N-1} \sum_{m=0}^{N-1}   \mu \biggl\{  \tilde{X}_n   \int_0^1 f''(   W_{n,m} + s Y_{n,m}     )[  Y_{n,m},  \theta_n] \,   - f''(   W_{n,m}   )[  Y_{n,m},  \theta_n ] \, ds
		\biggr\}  \\
		&+ \sum_{n=0}^{N-1} \sum_{m=0}^{N-1} 
		\mu \biggl\{  \tilde{X}_n   f''(   W_{n,m}    )[  Y_{n,m},  \theta_n]  \biggr\}\\
		&= E_1 + E_2 +  \sum_{n=0}^{N-1} \sum_{m=0}^{N-1}  \sum_{ \substack{|i - n| = m \\ 
				0 \le i < N } } 
		\mu \biggl\{  \tilde{X}_n   \tilde{X}_i   f''(   W_{n,m}    )[   \theta_i, \theta_n  ]  \biggr\}.
	\end{align*}
	
	Recalling $\sigma_{n,m} = \mu(X_n X_m)$, we see that 
	\begin{align*}
		&\mu \biggl\{  B^{-1} \sum_{n=0}^{N-1}  \sum_{m=0}^{N-1}    \sigma_{n,m}   f''(W) [ \theta_n, \theta_m ]  \biggr\} 
		=  \sum_{n=0}^{N-1}  \sum_{m=0}^{N-1}  \sum_{ \substack{|i - n| = m \\ 
				0 \le i < N } } 
		\mu \biggl\{   \mu( \tilde{X}_n
		\tilde{X}_i  )  f''(W) [  \theta_i , \theta_n ]  \biggr\}.
	\end{align*}
	Thus, what remains from $-I$ after subtracting $E_1 + E_2$ is 
	\begin{align}
		&-I - E_1 - E_2 \notag  \\
		&=\sum_{n=0}^{N-1}  \sum_{m=0}^{N-1}  \sum_{ \substack{|i - n| = m \notag  \\ 
				0 \le i < N } } \mu \biggl\{ 
		\tilde{X}_n   \tilde{X}_i   f''(   W_{n,m}    )[   \theta_i, \theta_n  ] - \mu( \tilde{X}_n
		\tilde{X}_i  )   f''(W) [ \theta_i, \theta_n ]  
		\biggr\} \\
		&=\sum_{n=0}^{N-1}  \sum_{m=0}^{N-1}  \sum_{ \substack{|i - n| = m  \label{eq:sunk-2} \\ 
				0 \le i < N } } \mu \biggl\{ 
		\tilde{X}_n   \tilde{X}_i   \overline{f''(   W_{n,m}    )[   \theta_i, \theta_n  ] } \biggr\} \\
		&- \sum_{n=0}^{N-1}  \sum_{m=0}^{N-1}  \sum_{ \substack{|i - n| = m \\ 
				0 \le i < N } } \mu( \tilde{X}_n
		\tilde{X}_m  )  \biggl(   \mu(   f''(W) [ \theta_i, \theta_n ] ) -   \mu(   f''(W_{n,m}) [ \theta_i, \theta_n ]  )  \biggr). \label{eq:sunk-3}
	\end{align}
	Since $\overline{f''(   W_{n,N-1}    )[   \theta_i,  \theta_n  ] }  = 0$, we can write
	\eqref{eq:sunk-2} as a telescopic sum:
	\begin{align*}
		\eqref{eq:sunk-2}
		&= \sum_{n=0}^{N-1}  \sum_{m=0}^{N-1}  \sum_{ \substack{|i - n| = m \\ 
				0 \le i < N } } \mu \biggl\{ 
		\tilde{X}_n   \tilde{X}_i  (  \overline{f''(   W_{n,m}    )[    \theta_i ,  \theta_n  ] }   
		-  \overline{f''(   W_{n,N-1}    )[    \theta_i ,  \theta_n  ] } ) 
		\biggr\} \\
		&= \sum_{n=0}^{N-1}  \sum_{m=0}^{N-1}   \sum_{k=m + 1}^{N-1}    \sum_{ \substack{|i - n| = m \\ 
				0 \le i < N } } \mu \biggl\{ 
		\tilde{X}_n   \tilde{X}_i  (  \overline{f''(   W_{n,k-1}   )[    \theta_i ,  \theta_n  ] }   
		-  \overline{f''(   W_{n,k}    )[    \theta_i,  \theta_n ] } ) \biggr\}  \\
		&= E_3 + E_4 + E_5.
	\end{align*}
	Finally, by writing 
	$$
	f''(W) [  \theta_i,  \theta_n ]  -    f''(W_{n,m}) [  \theta_i,  \theta_n ]  
	= \sum_{k=0}^m (   f''(W_{n,k-1} ) [  \theta_i,  \theta_n ]  -    f''(W_{n,k}) [  \theta_i ,  \theta_n ]       ),
	$$
	we find that 
	\begin{align*}
		\eqref{eq:sunk-3} 
		&= - \sum_{n=0}^{N-1}  \sum_{m=0}^{N-1} \sum_{k=0}^m  \sum_{ \substack{|i - n| = m \\ 
				0 \le i < N } } \mu( \tilde{X}_n
		\tilde{X}_m  )  \biggl(   \mu(  f''(W_{n,k-1}) [  \theta_i ,  \theta_n ] ) -   \mu(   f''(W_{n,k}) [  \theta_i,  \theta_n ]  )  \biggr) \\
		&= E_6 + E_7.
	\end{align*}
\end{proof}

Applying (A1)-(A3) to estimate each  $E_i$ separately for $1 \le i \le 7$, we obtain the following:

\begin{prop}\label{prop:estim_I} For any $g \in \sM$,
\begin{align}\label{eq:estim_I_stein}
I \le \sum_{i=1}^7 |E_i| \lesssim 
	B^{-3/2} N  L^3 \Vert g \Vert_{\sM} + 
	(C_1 L  + C_2 + C_3) B^{-3/2} N  \Vert g \Vert_\sM  \sum_{m=1}^{N-1} m \rho(m).
\end{align}
\end{prop}

The proof of Proposition \ref{prop:estim_I} is similar to \cite[Theorem 2.1]{LS20} and therefore is deferred until Appendix \ref{sec:proof_estimates_Ei}.

\subsubsection{Estimate on $\mu[g(W)] - \bfE[g(Z)]$}\label{s:55}

Combining \eqref{eq:decomp}, \eqref{eq:estim_II-2}, and \eqref{eq:estim_I_stein}, we arrive at the upper bound
\begin{align*}
	&|  \mu[g(W)] - \bfE[g(Z)] | \le I + II \\
	&\lesssim B^{-3/2} N  L^3 \Vert g \Vert_{\sM} + 
	(C_1 L  + C_2 + C_3) B^{-3/2} N  \Vert g \Vert_\sM  \sum_{m=1}^{N-1} m \rho(m) \\
	&+ \Vert g \Vert_{\sM}C_0 C_1^{3/2} B^{-3/2} N C_\rho^{1/2}\biggl[ C_\rho 
	+  \sum_{j=1}^{N-1} j^{1/2} \rho(j) \biggr] \\
	&\lesssim 
	B^{-3/2} N \Vert g \Vert_{\sM} \biggl[ 
	L^3 + (C_1 L  + C_2 + C_3) \sum_{m=1}^{N-1} m \rho(m)+C_0 C_1^{3/2}   C_\rho^{1/2}\biggl[ C_\rho 
	+  \sum_{j=1}^{N-1} j^{1/2} \rho(j) \biggr]
	\biggr],
\end{align*}
where $C_\rho = \sum_{j=0}^{N-1} \rho(j)$.
This completes the proof of Theorem \ref{thm:prelim}.

\subsection{Verification of the conditions of Theorem \ref{thm:prelim} under \eqref{eq:fcb}}\label{s:56}
We derive the final upper bound in Theorem \ref{thm:main} by verifying that \eqref{eq:fcb} implies (A1)-(A3). 

Note that, using \eqref{eq:d2-lip} and the definition of $\Vert \cdot \Vert_{\sM}$ in Proposition 
\ref{prop:M_norm}, we obtain the following bounds for any $w_1, w_2, \hat{w}_1, \hat{w}_2 \in D$, $s,t,u \in [0,1]$ and $0 \le i,n < N$:
\begin{align}\label{eq:bounds_Hg}
\begin{split}
	&| H_g(w_1, w_2; s,t,u, i, n) | \lesssim \Vert g \Vert_{\sM} B^{-1/2} \Vert w_2 \Vert_\infty, \\
	&| H_g(w_1, w_2; s,t,u, i, n) - H_g( \hat{w}_1 , \hat{w}_2; s,t,u, i, n)  | \\
	&\lesssim \Vert g 
	\Vert_{\sM} B^{-1/2} ( \Vert w_1 - \hat{w}_1 \Vert_\infty + \Vert w_2 - \hat{w}_2 \Vert_\infty ).
\end{split}
\end{align}

\begin{prop}\label{prop:0118} Assume that $(X_n)_{0 \le n < N}$ satisfies \eqref{eq:fcb}. 
	Then, (A1)-(A3) in Theorem \ref{thm:prelim} hold with
	\begin{align*}
		\rho(n) = \hat{R}(n) , \quad C_1  = C C_* L^2  , \quad C_i = C ( C_*  + 1) (L + 1)^3, \quad i = 2,3,
	\end{align*}
	where $C > 0$ is an absolute constant, $\hat{R}(n) = \max_{n \le j \le N} R(j)$ for $n \ge 1$, and $\hat{R}(0) = 1$. Consequently, for any 
	$g \in \sM_0(C_0)$,
	$$
	| \mu [ g(W_N) ] - \bfE[g(Z)] |  \le C  C_{N}  B^{-3/2} N \Vert g \Vert_{\sM},
	$$
	where $C > 0$ is an absolute constant,
	$$
	C_{N} = L^3 + ( C_*  + 1) (L + 1)^3 \hat{R}_3(N)
	+ C_0 L^3 C_*^{3/2}  \hat{R}_1(N) ^{1/2}(\hat{R}_1(N)+  \hat{R}_2(N)),
	$$
	and
	$$
	\hat{R}_1(N) = \sum_{j=0}^{N-1} \hat{R}(j), \quad \hat{R}_2(N) = \sum_{j=0}^{N-1} j^{1/2} \hat{R}(j), \quad 
	 \hat{R}_3(N) = \sum_{j=0}^{N-1} j \hat{R}(j).
	$$
\end{prop}

\begin{proof} (A1): Given integers $0 \le n < m < N$, define:
	$$
	I = \{n,m\}, \quad  I_1 = \{ n \}, \quad I_2 = \{ m \}, \quad F(x,y) = xy.
	$$
	Then, $\Vert F \Vert_{ \text{Lip} } \lesssim L^2$ and, since $\mu(X_n)  = 0$,
	\begin{align*}
		\mu(X_n X_m) &= \int_M F(X_n(x), X_m(x))  \, d \mu(x) - \iint_{M \times M}F(X_n(x), X_m(y))  \, d \mu(x) \, d \mu(y) \\
		&= \int F \, d \nu_{I} 
		- \int F d ( \nu_{I_1} \otimes \nu_{I_2} ).
	\end{align*}
	Thus, it follows from \eqref{eq:fcb} that
	\begin{align*}
		| \mu(X_n X_m) | \lesssim  C_* L^2 R(n - m).
	\end{align*}
	for all $0 \le n , m < N$ with $R(0):=1$.
	
	\noindent (A2): In what follows, for any subset $J \subset \bZ_+$, we denote by $x_J$ 
	a general vector $(x_{n_1}, \ldots x_{n_{|J|} } )$ in $[-L, L]^{|J|}$ where $n_1 < \ldots < n_{|J|}$ is an enumeration
	of the elements of $J$ in increasing order. 
	
	Let $0 \le i,k,n,m < N$ be integers with
	$|i -n | = m > 0$ and $m \le k < N$. Given $g \in \sM_0$, $s,t,u \in [0,1]$, we define 
	\begin{align}\label{eq:def_F}
		F( x_I ) =  \varphi(x_i, x_n) G( x_{  J_{0, n - k }   },  x_{  J_{n + k, N - 1} }  ),
	\end{align}
	where $I = \{i,n\} \cup J_{0, n - k } \cup J_{n + k, N - 1} $, 
	$J_{ i_1, i_2 } = [i_1, i_2] \cap \bZ_+$, 
	$$
	\varphi(x_i, x_n) = x_ix_n \quad \text{and} \quad 
	G( x_{  J_{0, n - k }   },  x_{  J_{n + k, N - 1} }  ) = H_g \biggl(  \sum_{ \substack{ 0 \le i < N \\  |i - n| > k  } }   \theta_i(t)  x_i,  
	\sum_{ \substack{ 0 \le i < N \\  |i - n| = k  } }   \theta_i(t)  x_i;  s,t,u,i, n
	\biggr).
	$$
	Using \eqref{eq:bounds_Hg}, we find that 
	$$
	\Vert G \Vert_{\Lip} \lesssim \Vert g \Vert_{\sM} B^{-1/2} L + \Vert g \Vert_{\sM} B^{-1/2}.
	$$
	Therefore, 
	\begin{align}\label{eq:F_lip}
		\Vert F \Vert_{ \text{Lip}  } \lesssim  (L + 1)^3 \Vert g \Vert_{\sM} B^{-1/2}.
	\end{align}
	Suppose $i = n + m$. Recall that \eqref{eq:fcb} implies
	 \eqref{eq:fcb2} corresponding to a functional correlation bound
	 with an arbitrary number of index sets $K$. In this case, we set 
	\begin{align*}
		&K = 3, \quad  I_1 = J_{0, n-k}, \quad  I_2 = \{ n \},  \quad I_3 = \{ i\}  \cup J_{ n + k, N - 1 }.
	\end{align*}
	Since $\mu(X_n) = 0$, we have 
	$$
	\int F \, d ( \nu_{I_1} \otimes \nu_{I_2} \otimes \nu_{I_3} ) = 0.
	$$
	Now, an application of \eqref{eq:fcb2} yields
	\begin{align}
		&\biggl|\mu \biggl\{  X_n X_i   H_g(  S_{n,k}, T_{n,k} ; s,t,u,i, n)  \biggr\} \biggr| = \biggl| \int F \,  d \nu_{I}   -   \int F \,  d ( \nu_{I_1} \otimes \nu_{I_2} \otimes \nu_{I_3} )    \biggr| \notag \\
		&\lesssim C_*  (L + 1)^3 \Vert g \Vert_{\sM} B^{-1/2} ( R(k) + R(m)  ) \notag  \\
		&\lesssim  (C_* + 1 ) (L + 1)^3 \Vert g \Vert_{\sM} B^{-1/2} \hat{R}(m). \label{eq:fcb_appli_1}
	\end{align}
	If $i = n - m$, we set
	\begin{align*}
		&K = 3, \quad  I_1 = J_{0, n-k} \cup \{ i\}, \quad  I_2 = \{ n \},  \quad I_3 = J_{ n + k, N - 1 }
	\end{align*}
	and again obtain \eqref{eq:fcb_appli_1} by an application of \eqref{eq:fcb2}. Finally, 
	if $m = 0$, \eqref{eq:fcb_appli_1} follows directly from $\Vert F \Vert_\infty \lesssim (L + 1)^3 \Vert g \Vert_{\sM} B^{-1/2}$.

	\noindent (A3): The verification of (A3) is similar to that of (A2), so we only provide an outline. 
	Let $0 \le i,k,n,m < N$ be integers with $|i -n | = m$, $2m + 1 \le k < N$.
	Given $g \in \sM_0$, $s,t,u \in [0,1]$, we define $F$ as in \eqref{eq:def_F}. Further, let 
	\begin{align*}
		&I = \{i,n\} \cup J_{0, n - k } \cup J_{n + k, N - 1}, \quad I_1 =  J_{0, n - k }, \quad I_2 = \{i,n\}, \quad  I_3 = J_{n + k, N - 1}.
	\end{align*}
	Then, 
	$$
	\mu \biggl\{  X_n X_i   \overline{H_g(  S_{n,k}, T_{n,k} ; s,t,u,i, n)}  \biggr\} = \int F \, d \nu_{I} - \int  \varphi  \, d \nu_{ I_2  } 
	\int G \, d \nu_{ I_1 \cup I_3  } = R_1 + R_2,
	$$
	where
	$$
	R_1 = \int F \, d \nu_I  -  \int  \varphi  \, d \nu_{ I_2 } 
	\int G \, d (  \nu_{ I_1  } \otimes  \nu_{  I_3 } ),
	$$
	and 
	$$
	R_2 =   \int  \varphi  \, d \nu_{ I_2 }  \biggl(  \int G \, d (  \nu_{ I_{1}  } \otimes  \nu_{ I_{3}  } ) - \int G d \nu_{I_1 \cup I_3 }    \biggr).
	$$
	Since
	$$
	\int  \varphi  \, d \nu_{ I_2 } 
	\int G \, d (  \nu_{ I_{1}  } \otimes  \nu_{ I_{3}  } ) =  \int F  \, d(  \nu_{I_{1}} \otimes   \nu_{I_{2}} \otimes   \nu_{I_{3}}  ),
	$$
	we can apply \eqref{eq:fcb2} separately to $R_1$ and $R_2$, resulting in the bound
	$$|R_i| \lesssim (C_* + 1 ) (L + 1)^3 \Vert g \Vert_{\sM} B^{-1/2} \hat{R}(k - m), \quad i = 1,2.$$	
\end{proof}

\section{Proof of Lemma \ref{lem:fcb_pikovsky}}\label{sec:proof_pikovsky}

Let $T$ be the map defined in \eqref{eq:pikovsky} with parameter $\gamma > 1$. Recall that $\lambda$ 
denotes the Lebesgue measure on $M = [-1,1]$ normalized to probability. Let $\cL$ denote 
the transfer operator associated with $(T, \lambda)$.

\noindent\textbf{Convention:} Throughout this section, $C_\gamma$ denotes a generic positive 
constant depending only on $\gamma$. The value of $C_\gamma$ may change from one display 
to the next.

\subsection{Preliminary} In this section, we gather some basic (well-known) properties of Pikovsky maps that will be used in the proof of Lemma~\ref{lem:fcb_pikovsky}. In particular, we will rely on the distortion bounds and mixing estimates stated in \eqref{eq:dist} and \eqref{eq:ml_pikovsky}, respectively.

Denote by $g_{-} : (-1,1) \to M_{-}$ and $g_{+} : (-1,1) \to M_{+}$ the left and right inverse branches of $T$, respectively, where $M_{-} = (-1,0)$ and $M_{+} = (0,1)$.

Following \cite{CLM23}, for each $n \ge 1$, define
\begin{align*}
	\Delta_n^{-} = T^{-1} ( \Delta_{n-1}^{-} ) \cap M_{-}  \quad \text{and} \quad 
	\Delta_n^{+} = T^{-1} ( \Delta_{n-1}^{+} ) \cap M_{+},
\end{align*}
where 
\begin{align*}
	\Delta_0^{-} = ( g_{-}(0), 0  ) = T^{-1}(M_{+}) \cap M_{-} \quad \text{and} \quad \Delta_0^{+} = ( 0, g_{+}(0)  ) = T^{-1}(M_{-}) \cap M_{+}.
\end{align*}
Then, $\{ \Delta_n^{-}  \}_{n \ge 1}$ is a (mod $\lambda$) partition of $M_{-}$, $\{ \Delta_n^{+}  \}_{n \ge 1}$ is a (mod $\lambda$) partition of $M_{+}$, and $T$ maps $\Delta_n^{\pm}$ bijectively onto $\Delta_{n-1}^{\pm}$. Moreover, we have 
$\Delta^+_\ell = ( x_{\ell }^+, x_{\ell + 1}^+  )$ and $\Delta^-_\ell = ( x_{\ell + 1 }^-, x_{\ell}^-  )$,
where $x_{\ell}^{+} =  g_{+}^{\ell}(0)$ and $x_{\ell}^{-} =  g_{-}^{\ell}(0)$.

Next, for each $n \ge 1$ define 
\begin{align*}
	\delta_n^{-} = T^{-1} (   \Delta_{n-1}^{+}   ) \cap \Delta_0^{-} \quad \text{and}  \quad 
	\delta_n^{+} = T^{-1} (   \Delta_{n-1}^{-}   ) \cap \Delta_0^{+}.
\end{align*}
Then, $\{ \delta_n^{-}  \}_{n \ge 1}$  and $\{ \delta_n^{+}  \}_{n \ge 1}$ are (mod $\lambda$) partitions of 
$\Delta_0^{-}$ and $\Delta_0^{+}$, respectively. Moreover, $\delta_n^{-} = (y_{n-1}^{-}, y_n^{-} )$ 
and $\delta_n^{+} = (y_{n}^{+}, y_{n-1}^{+})$ where 
$y_n^{-} = g_{-}( x_{n-1}^{+} )$ and $y_n^{+} = g_{+}( x_{n-1}^{-} )$.
Note that the following maps are all bijective:
\begin{align*}
	&T : \delta_n^{-} \to \Delta_{n-1}^{+}, \quad T : \delta_n^{+} \to \Delta_{n-1}^{-}, \\ 
	&T^n : \delta_n^{-} \to \Delta_{0}^{+},  \quad T^n : \delta_n^{+} \to \Delta_{0}^{-}.
\end{align*}

By \cite[Lemma 2]{CHMV10}, the following estimates hold:
\begin{align}\label{eq:estim_partition}
	\begin{split}
	&1 - x_n^{+} \le C_\gamma n^{- \frac{1}{\gamma - 1}}, \quad x_n^{-} + 1 \le C_\gamma n^{- \frac{1}{\gamma - 1}}, \\
	&|\Delta_n^{\pm}| \le C_\gamma 
	n^{ - \frac{\gamma}{\gamma - 1}  }, \quad |y_n^{\pm}| \le C_\gamma n^{ - \frac{\gamma}{\gamma - 1}  }, \\
	&|\delta_n^{\pm}| \le |y_n^{\pm}| +  |y_{n+1}^{\pm}| \le C_\gamma n^{ - \frac{\gamma}{\gamma - 1}  }.
	\end{split}
\end{align}

Set $Y = \Delta_0^{-} \cup  \Delta_0^{+}$, and define the return time function
$$
\tau(x) = \inf \{  n \ge 1 \: : \:  T^n(x) \in Y \}, \quad x \in M.
$$
Define a partition of $M$ by 
\begin{align*}
	\cP = \{  \delta_{n}  \}_{n \ge 1} 
	\cup \{  \Delta_{n}  \}_{n \ge 1},
\end{align*}
where $\delta_n := \delta_{n}^{-} \cup \delta_{n}^{+} $ and $\Delta_n := \Delta_{n}^{-} \cup  \Delta_{n}^{+}$.
For each $a \in \cP$, the restriction $\tau |_a$ is a constant function whose value we denote by 
$\tau(a) \in \{1,2,\ldots\}$. Let $F_a : a \to Y$ be the first return map $F_a := T^{ \tau(a) }$, and let 
$m$ denote Lebesgue measure on $Y$, normalized to probability.

\begin{lem} There exists a constant $\fD > 0$ depending only on $\gamma$ such that the following holds for all $0 \le m < n$: Let $a \in \cP$ with $\tau(a) = n$. Then, for all $x, x' \in a$:
	\begin{align}\label{eq:dist}
		| 
		\log (T^{ n - m })' (  T^m )(x) - \log (T^{ n - m })' (  T^m )(x')
		| \le \fD | T^n(x) - T^n(x') |. 
	\end{align}
\end{lem}

\begin{proof} There are two cases: (1) $x, x' \in \delta_n^{\sigma}$ or $x, x' \in \Delta_n^{\sigma}$ where 
	$\sigma \in \{-, +\}$; (2) $x \in \delta_n^{\sigma_1}, x' \in \delta_n^{\sigma_2}$ or 
	$x \in \Delta_n^{\sigma_1}, x' \in \Delta_n^{\sigma_2}$ where $\sigma_1 \neq \sigma_2 \in \{-, +\}$. 
	The first case follows directly from \cite[Proposition 3.10]{CLM23}. The second case can be 
	treated as  in \cite[Remark 1]{CHMV10}, by exploiting symmetricity of the map \eqref{eq:pikovsky}.
\end{proof}

For a map $\psi : Y \to [0, \infty)$, define 
$$
|\psi |_{\LL} = \sup_{y \neq y' \in Y} \frac{| \log \psi(y) - \log \psi(y')  |}{|y - y'|},
$$
where we adopt the conventions $\log 0 = - \infty$ and $\log 0 - \log 0 = 0$. In the sequel, for 
a nonnegative measure $\mu$ on $Y$, we often
write $|\mu|_{\LL}$ to denote $|d \mu / d m|_{\LL}$, 
with the convention that the density of $\mu$ is fixed.

\begin{prop}\label{prop:gm} The first return map $F_a$ is a full-branch mixing Gibbs--Markov map. More precisely, the following hold for some constants $\Lambda > 1$ and $K, \delta_\# > 0$ depending only on $\gamma$.
\begin{itemize}
	\item[(i)] $F_a$ is bijective, and
	\begin{align}\label{eq:unif_exp}
	d( F_a(y), F_a(y')  ) \ge \Lambda d(y,y') \quad \forall y,y' \in a, \quad a \subset Y, \: a \in \cP.
	\end{align}
	Further, $F_a$ is non-singular with log-Lipschitz Jacobian:
	\begin{align}\label{eq:loglip}
	\zeta = \frac{d (F_a)_* (m|_a)}{dm}
	\quad \text{satisfies} \quad
	|\zeta|_{\LL} \leq K
	.
	\end{align}
	\item[(ii)] $(T^n)_*m(Y) \ge \delta_\#$ for all $n \ge 1$. 
\end{itemize}
In addition, $\tau$ is integrable with respect to $m$.
\end{prop}

\begin{proof} \textbf{(i):} \eqref{eq:unif_exp} is easy to see by inspecting the definition \eqref{eq:pikovsky}, while 
\eqref{eq:loglip} is a direct consequence of \eqref{eq:dist}.

\noindent \textbf{(ii):} By \eqref{eq:estim_partition}, for each $j,N \ge 1$, we have 
$$
(T^j)_*m(  \tau \ge N  ) \le \frac{(T^j)_* \lambda (  \tau \ge N  )}{  \lambda(Y) }
= \frac{ \lambda (  \tau \ge N  )}{  \lambda(Y) }
\le C_\gamma N^{ - 1 / ( \gamma - 1 ) }.
$$
Moreover, $m(  \tau = 1 ) \ge m( \delta_1^{-} ) \ge C_\gamma > 0$. Thus, (ii) follows\footnote{This result is applied in the case of the constant sequence $T_1 = T_2 = \ldots = T$.
} by \cite[Proposition 2.1]{KL24}.

Finally, using \eqref{eq:estim_partition}, we compute 
$$
\int \tau d m = \sum_{n=1}^\infty m ( \tau \ge n ) = \sum_{n=1}^\infty m \biggl(  \bigcup_{\ell \ge n} 
\delta_{\ell}^{ \pm }
\biggr) \le \sum_{n=1}^\infty (  |y_{n}^{-}| + |y_{n}^{+}| ) < \infty. 
$$
\end{proof}

It follows 
 from Proposition \ref{prop:gm} 
(see
\cite[Proposition 3.4]{KL21} and 
\cite[Proposition 3.1]{KKM19}) that there exist constants $0 < K_1 < K_2$ depending only 
on $\gamma$, such that whenever $\mu$ is a nonnegative measure on $Y$ with
$| \mu |_{\LL} \le K_2$, then for each $a \in \cP$, $a \subset Y$, 
\begin{align}\label{eq:K1}
|  (  F_a  )_*(  \mu |_{ a } )  | \le K_1.
\end{align}
Moreover, $K_1,K_2$ can be chosen arbitrarily large. Fix such $K_1,K_2$. Then, we have the 
following mixing estimate:

\begin{lem}\label{lem:mixing_pikovsky} Let $\mu$ be a probability measures on $M$ with density $h = d \mu / d \lambda$. Suppose 
that the following hold:
\begin{align}\label{eq:regular_measure}
|   (  T^{ n } )_*(  \mu |_{ \{ \tau = n  \} } )  |_{\LL} \le K_1 \quad \forall n \ge 1,
\end{align}
and 
\begin{align}\label{eq:tail_bound}
	\mu(  \tau \ge n  ) \le C_\beta n^{- \beta } \quad \forall n \ge 1,
\end{align}
where $\beta \le \gamma/( \gamma - 1)$.
Then,
\begin{align}\label{eq:ml_pikovsky}
\Vert \cL^n (h)  - 1  \Vert_{ L^1( \lambda ) } \le C n^{  - \beta },
\end{align}
where $C$ is a constant depending only on $K_1,K_2, C_\beta$ and $\gamma, \beta$.
\end{lem}

\begin{proof} The desired estimate follows as an immediate consequence of \cite[Theorem 2.6]{KL24}, applied in the special case $T_1 = T_2 = \ldots = T$. To see this, it suffices to observe the following:
\begin{itemize}
	\item In the terminology of \cite{KL24}, \eqref{eq:regular_measure} and \eqref{eq:tail_bound} together mean that that $\mu$ is a regular measure with tail bound $r(n) := C_\beta n^{- \beta }$.
	\item Assumptions (NU:1-5) in \cite[Section 2.1]{KL24} are satisfied by Proposition \ref{prop:gm}.
	\item By \eqref{eq:estim_partition}, 
	$m( \tau \ge n ) \le C_\gamma  n^{ -  \gamma/( \gamma - 1)  }$.
\end{itemize}
\end{proof}

\begin{remark} Estimate \eqref{eq:ml_pikovsky} is used only in \eqref{eq:ml_appli} to obtain
$\Vert \cL^m ( v  ) - 1 \Vert_{L^1(\lambda)} = O(m^{ - 1 / ( \gamma - 1 ) })$ for the function 
$v = \cL^{\ell_p}(h_{ \ell_0 \ldots \ell_p })$ where $h_{ \ell_0 \ldots \ell_p }$ is defined in \eqref{eq:h_ls}.
Alternatively, the latter estimate follows from \cite[Theorem~2]{Y99} combined with the results of \cite{CM06}, 
since, in the notation of \cite{Y99}, $v$ lifts to a function in $\cC_\beta^+(\Delta)$.
\end{remark}

\subsection{Cylinders} Define 
$$
L(n) = \sup \{ p \ge 0  \: : \: \tau^p \le n \},
$$
where $\tau^0 = 0$ and 
$
\tau^j = \tau^{j-1} + \tau \circ T^{  \tau^{j-1} }
$ for $j \ge 1$. This quantity
represents the number of returns to $Y$ by time $n$. Given 
$n \ge 1$, we consider the (mod $\lambda$) partition of $M$ into the cylinder sets
\begin{align}\label{eq:cylinder}
	C_n(\ell_0, \ldots, \ell_p) = \{  \tau_1 = \ell_0, \ldots, \tau_{p} = \ell_{p} \} \cap \{L(n) = p\},
\end{align}
where $1 \le \ell_0 < \ldots < \ell_{p - 1} \le n < \ell_{p}$ if $1 \le p \le n$, and $\ell_0 > n$ if $p = 0$. Here,
$\ell_0 < \ldots < \ell_{p-1}$ enumerate returns to $Y$ by time $n$, and $\ell_{p}$ is the first return time 
after $n$. By definition, we have $\tau^{ L(n)  } \le n < \tau^{ L(n) + 1 }$.
We can express $C_n(\ell_0, \ldots, \ell_{p})$ using elements of the partition $\cP$ as follows:
\begin{align*}
	C_{n}(\ell_0, \ldots, \ell_{p}) &=  ( \delta_{\ell_0} \cup \Delta_{\ell_0} ) \cap  \bigcap_{j=0}^{p-1}  T^{- \ell_j } \delta_{\ell_{j+1} - \ell_j}, \quad \text{for $p > 0$},  \\
	C_{n}( \ell_ 0 ) &= \delta_{\ell_0} \cup \Delta_{\ell_0}, \quad \text{for $p = 0$}.
\end{align*}

Using \eqref{eq:estim_partition} and  \eqref{eq:dist}, we can control the size of 
cylinder sets:

\begin{prop} For all $p \ge 1$:
	\begin{align}\label{eq:cylinder_estim-1}
		\begin{split}
			&\lambda(C_{n}(\ell_0)) \le  C_\gamma \ell_0^{  - \gamma / ( \gamma - 1  ) }, \\
			&\lambda(  C_{n}(\ell_0, \ldots, \ell_p)  ) \le  \min\{  \Lambda^{-1},  C_\gamma \ell_0^{ - \gamma /( \gamma - 1 )  } \}
			\prod_{j=0}^{p-1} \min \{   \Lambda^{-1}, C_\gamma ( \ell_{j+1} - \ell_j  )^{ - \gamma /( \gamma - 1 )  } \}.
		\end{split}
	\end{align}
	In particular, for all $p \ge 0$, 
	\begin{align}\label{eq:cylinder_estim-2}
		\lambda (  C_{n}(\ell_0, \ldots, \ell_p)  )  \le C n^{ - \gamma /( \gamma - 1 )  }.
	\end{align}
\end{prop}

\begin{proof} The first upper bound in \eqref{eq:cylinder_estim-1} follows readily from \eqref{eq:estim_partition}. To show the second bound in \eqref{eq:cylinder_estim-1},
let $x, x' \in C_{n}(\ell_0, \ldots, \ell_p)$, $p \ge 1$. For $-1 \le j < p$ with $\ell_{-1} = 0$, we have 
	\begin{align*}
	| T^{\ell_j} (x) - T^{\ell_j}  ( x' ) | \le \sup_{ \xi \in Y} \frac{1}{ (  T^{ \ell_{j+1}  - \ell_j }  )' \circ ( T^{ \ell_{j+1} - \ell_j } |_{\delta_{ \ell_{j+1} - \ell_j }  } )^{-1}(\xi) }
	| T^{\ell_{j+1}} (x) - T^{\ell_{j+1}} (x') |.
	\end{align*}
	Combining \eqref{eq:unif_exp}, \eqref{eq:dist}, we find that 
	$$
	\sup_{ \xi \in Y} \frac{1}{ (  T^{ \ell_{j+1} - \ell_j }  )' \circ ( T^{ \ell_{j+1} - \ell_j } |_{\delta_{ \ell_{j+1} - \ell_j }  } )^{-1}(\xi) } \le 
	\min\{ \Lambda^{-1}, C_\gamma \lambda(  \delta_{ \ell_{j+1} - \ell_j } )   \}.
	$$
	Further, by \eqref{eq:estim_partition}, we have that 
	$\lambda(  \delta_{ \ell_{j+1} - \ell_j } ) \le C_\gamma ( \ell_{j+1} - \ell_j )^{ - 1 / (  \gamma - 1 )}$. 
	Therefore, 
	\begin{align}\label{eq:ind_cylinder}
		\begin{split}
			&| T^{\ell_j} (x) - T^{\ell_j}  ( x' ) |  \\
			&\le \min \{  \Lambda^{-1},  C_\gamma  ( \ell_{j+1} - \ell_j )^{ - 1 / (  \gamma - 1 )} \} 
			| T^{\ell_{j+1}} (x) - T^{\ell_{j+1}} (x') |.
		\end{split}
	\end{align}
	The second upper bound in \eqref{eq:cylinder_estim-1} follows by iterating 
	\eqref{eq:ind_cylinder} $p + 1$ times.	
	
	It remains to verify \eqref{eq:cylinder_estim-2}. For $p = 0$, this follows from
	the first bound in \eqref{eq:cylinder_estim-1}, since $\ell_p > n$.
	If $p > 0$, observe that for a given cylinder 
	$C_{n}(\ell_0, \ldots, \ell_p)$, we must have either  $\ell_0 \ge n/(p+1)$ or 
	$\ell_{j+1} - \ell_j \ge n/(p+1)$ for some $0 \le j < p$. Thus, by the second upper bound 
	in \eqref{eq:cylinder_estim-1}, 
	$$
	\lambda (  C_{n}(\ell_0, \ldots, \ell_p)  ) \le C_\gamma \Lambda^{ - ( p - 1) } (p + 1)^{  1 / ( \gamma - 1)  }
	n^{ - 1 / ( \gamma - 1)} \le C_\gamma n^{ - 1 / ( \gamma - 1)}.
	$$
\end{proof}

\begin{prop}\label{prop:image_cylinder} If $0 \le \ell < n$, then
	\begin{align}\label{eq:cylinder_image}
	\lambda( T^{\ell}( C_n(\ell_0, \ldots, \ell_p) ) ) \le C_\gamma (n - \ell)^{ - \gamma / ( \gamma - 1) }.
	\end{align}
	
\end{prop}

\begin{proof} The case $\ell = 0$ is clear from \eqref{eq:cylinder_estim-2}. If $p = 0$ and $\ell > 0$, 
we have $T^{ \ell } (  C_{n}(\ell_0) ) \subset  \Delta_{\ell_0 - \ell}$, so the desired bound 
follows from \eqref{eq:estim_partition}. Now, consider the case $p > 0$. We observe that 
	\begin{align*}
		&T^\ell ( C_{n}(\ell_0, \ldots, \ell_p )  ) \\
		&\subset 
		\begin{cases}
			\{  \tau^1 = \ell_{j+1} - \ell, \ldots,  \tau^{ p + 1 - j } = \ell_p - \ell   \}, & \ell_j \le \ell < \ell_{j+1}, \: 0 \le j < p, \\
			\{  \tau^{1 } = \ell_0 - \ell, \ldots,   \tau^{p+1} = \ell_p - \ell \}, &\ell < \ell_0,
		\end{cases}\\
		&=\begin{cases}
			C_{n - \ell}( \ell_{j + 1} - \ell, \ldots, \ell_p - \ell ), & \ell_j \le \ell < \ell_{j+1}, \: 0 \le j < p, \\
			C_{n - \ell}( \ell_{0} - \ell, \ldots, \ell_p - \ell ), &\ell < \ell_0,
		\end{cases}
	\end{align*}
	where $\ell_{p-1} - \ell  \le n - \ell$ and 
	$\ell_p - \ell > n - \ell$. Therefore, \eqref{eq:cylinder_image} follows by 
	\eqref{eq:cylinder_estim-2}.
\end{proof}

\subsection{Verification of \eqref{eq:fcb}} We are now in a position to prove Lemma~\ref{lem:fcb_pikovsky}.
Let $0 \le i_1 < \ldots < i_n$ and $1 \le m < n$ be integers. Let $F : [ - \Vert f \Vert_\infty, \Vert f \Vert_\infty ]^n \to \bR$ be a separately Lipschitz continuous function, 
and define $H : M^2 \to \bR$ by
$$
H(x,y) = F( X_{i_1}(x), \ldots, X_{ i_m  }(x),   X_{ i_{ m + 1 }  }(y),  \ldots, X_{i_n}(y) ).
$$
Set
$$
i_* = \round{( i_{m + 1}  - i_m )/ 3} + i_m,
$$
and, without loss of generality, assume that $i_{m + 1}  - i_m  > 6$.
To establish Lemma \ref{lem:fcb_pikovsky}, we need to show that
\begin{align}
	\label{eq:toshow_fcb}
	\begin{split}
	&\biggl| \int_{  M  } H (x,x) \, d \lambda(x)  - \iint_{  M^2  } H (x,y) \, d \lambda(x) \, d \lambda(y) \biggr| \\ 
	&\le C_\gamma (  \Vert f \Vert_{ \Lip } + 1 )\Vert F \Vert_{ \Lip } (  i_{m + 1}  - i_m  )^{ - 1 / (  \gamma - 1 ) }.
	\end{split}
\end{align}
The proof consists of three steps.

\noindent\textbf{Step 1:} We decompose 
$$
M = \bigcup_{p = 0}^{ i_* } \bigcup_{\ell_0 < \ldots < \ell_p} C_{ i_* }(\ell_0, \ldots, \ell_p),
$$
and 
approximate 
$
x \mapsto H(x, y)
$
by a constant function on each cylinder. Here, the second union is taken over all integers 
$\ell_0 < \ldots < \ell_p$ as in \eqref{eq:cylinder}. 
For simplicity, in the sequel we omit the subscript 
$i_*$ and write $C(\ell_0, \ldots, \ell_p) = C_{ i_*}(\ell_0, \ldots, \ell_p)$.

For any $x,x' \in C(\ell_0, \ldots, \ell_p)$, $y \in M$:
\begin{align*}
	|H(x,y) - H(x',y)| &\le \sum_{j = 1}^m  \Vert F  \Vert_{ \Lip }  |  f(  T^{ i_j} x ) - f(  T^{ i_j} x' )   | \\
	&\le C_\gamma   \Vert F  \Vert_{ \Lip } \Vert f \Vert_{ \Lip }   \sum_{ \ell = 1 }^{ i_m  }  \lambda( T^{ \ell }  
	C(\ell_0, \ldots, \ell_p)
	 ) \\
	 &\le C_\gamma   \Vert F  \Vert_{ \Lip } \Vert f \Vert_{ \Lip }   \sum_{ \ell = 1 }^{ i_m  }  ( i_* - \ell  )^{ - \gamma / ( \gamma - 1) } \\
	 &\le C_\gamma   \Vert F  \Vert_{ \Lip } \Vert f \Vert_{ \Lip } (  i_{m+1} - i_m )^{ - 1/( \gamma - 1)  },
\end{align*}
where \eqref{eq:cylinder_image} was used in the second-to-last inequality. Therefore, 
\begin{align}\label{eq:step_1}
	\begin{split}
	&\int_{  M  } H (x,x) \, d \lambda(x)  - \iint_{  M^2  } H (x,y) \, d \lambda(x) \, d \lambda(y) \\ 
	&= \sum_{p=0}^{i_*} \sum_{\ell_0 < \ldots < \ell_p} \int_{ C(\ell_0, \ldots, \ell_p) } \biggl[
	H( c_{\ell_0 \cdots \ell_p}, x )
	-  \int   H( c_{\ell_0 \cdots \ell_p}, y ) \, d \lambda (y)  \biggr] \, d\lambda(x) \\ 
	&+O(    \Vert F  \Vert_{ \Lip } \Vert f \Vert_{ \Lip } (  i_{m+1} - i_m )^{ - 1/( \gamma - 1)  }  ),
	\end{split}
\end{align}
where the constant in the error term depends only on $\gamma$.

\noindent\textbf{Step 2:} We discard from the sum \eqref{eq:step_1} cylinders that 
are very small and approximate the error. To this end, note that for any $\ell > i_*$,
$$
\{  \tau^{ L(i_*) + 1 } > \ell  \} \subset \{  \tau(  T^{i_*} )  > \ell - i_* \} = T^{-i_*} \{ \tau > \ell - i_*  \}.
$$
Let 
$$
\ell_\# = i_{m} + \round{  2 ( i_{m+1}  - i_m  ) /3  }.
$$
Since $T$ preserves $\lambda$,
\begin{align}\label{eq:step2-1}
	\begin{split}
	&\sum_{p=0}^{i_*} \sum_{\ell_0 < \ldots < \ell_p, \, \ell_p > \ell_\#} 
	\lambda ( C(\ell_0, \ldots, \ell_p)  ) = \lambda ( \{   \tau^{ L(i_*) + 1 } > \ell_\# \}) \\
	&\le 
	(T^{n_*})_*\lambda(   \tau > \ell_\# - i_*  )\\
	& = \lambda( \tau >  \round{ 2( i_{m + 1}  - i_m ) / 3} -  \round{ ( i_{m + 1}  - i_m ) / 3}   ) \\
	& \le C_\gamma (  i_{m+1} - i_m )^{-1 / (\gamma - 1)},
	\end{split}
\end{align}
where \eqref{eq:estim_partition} was used in the last inequality. 

Combining \eqref{eq:step_1} and \eqref{eq:step2-1}, we arrive at the estimate
\begin{align}\label{eq:step_2-2}
	\begin{split}
		&\int_{  M  } H (x,x) \, d \lambda(x)  - \iint_{  M^2  } H (x,y) \, d \lambda(x) \, d \lambda(y) \\ 
		&= \sum_{p=0}^{i_*} \sum_{\ell_0 < \ldots < \ell_p, \, \ell_p \le \ell_{\#}} \int_{ C(\ell_0, \ldots, \ell_p) } \biggl[
		H( c_{\ell_0 \cdots \ell_p}, x )
		-  \int   H( c_{\ell_0 \cdots \ell_p}, y ) \, d \lambda (y)  \biggr] \, d\lambda(x) \\ 
		&+O(    \Vert F  \Vert_{ \Lip } (  \Vert f \Vert_{ \Lip }  + 1) (  i_{m+1} - i_m )^{ - 1/( \gamma - 1)  }  ).
	\end{split}
\end{align}

\noindent\textbf{Step 3:} Fix $0 < \ell_0 < \ldots < \ell_{p-1} \le i_* < \ell_p$ where $\ell_p \le \ell_\#$. 
(If $p = 0$, then $\ell_0 \le \ell_\#$.)
We replace $\lambda |_{  C(\ell_0, \ldots, \ell_p) } / \lambda (C(\ell_0, \ldots, \ell_p) )$ with 
$\lambda$ in the remaining integral and control the error using 
\eqref{eq:ml_pikovsky}. Set 
\begin{align}\label{eq:h_ls}
h_{  \ell_0 \cdots \ell_p } = \frac{ \mathbf{1}_{  C( \ell_0 \cdots \ell_p ) } }{ \lambda( C(\ell_0 \cdots \ell_p) ) },
\end{align}
and denote by $\tilde{H}( c_{\ell_0 \cdots \ell_p}, x  )$ the function that satisfies 
$\tilde{H}( c_{\ell_0 \cdots \ell_p},  T^{ i_{m + 1} } x  ) = H( c_{\ell_0 \cdots \ell_p},  x  )$.
Then,
\begin{align*}
	\cI &:= \int_{ C(\ell_0, \ldots, \ell_p) } \biggl[
	H( c_{\ell_0 \cdots \ell_p}, x )
	-  \int_M   H( c_{\ell_0 \cdots \ell_p}, y ) \, d \lambda (y)  \biggr] \, d\lambda(x) \\
	&=  \lambda(  C(\ell_0, \ldots, \ell_p) ) \int_M  \biggl[
	\tilde{H}( c_{\ell_0 \cdots \ell_p}, x )
	-  \int_M   H( c_{\ell_0 \cdots \ell_p}, y ) \, d \lambda (y)  \biggr] \circ T^{ i_{m+1} } 
	\cdot h_{  \ell_0 \cdots \ell_p } (x)  \, d \lambda (x) \\
	&=  \lambda(  C(\ell_0, \ldots, \ell_p)  ) \int_M  \biggl[
	\tilde{H}( c_{\ell_0 \cdots \ell_p}, x )
	-  \int_M   H( c_{\ell_0 \cdots \ell_p}, y ) \, d \lambda (y)  \biggr] \circ T^{ i_{m+1} } 
	\cdot ( h_{  \ell_0 \cdots \ell_p } - 1)(x)  \, d \lambda (x) \\
	&=  \lambda(  C(\ell_0, \ldots, \ell_p) ) \int_M  \biggl[
	\tilde{H}( c_{\ell_0 \cdots \ell_p}, x )
	-  \int_M   H( c_{\ell_0 \cdots \ell_p}, y ) \, d \lambda (y)  \biggr] 
	\cdot  \cL^{ i_{m+1} } ( h_{  \ell_0 \cdots \ell_p } - 1)(x)  \, d \lambda (x).
\end{align*}
Consequently,
\begin{align*}
	\cI \le   \Vert F \Vert_{ \Lip }   \lambda(  C(\ell_0, \ldots, \ell_p) )  
	\Vert    \cL^{ i_{m+1} - \ell_p } \cL^{  \ell_p } ( h_{  \ell_0 \cdots \ell_p } ) - 1  \Vert_{L^1( \lambda )}.
\end{align*}
The density $\cL^{\ell_p} h_{  \ell_0 \cdots \ell_p }$ is supported on $Y$, and
$$
\cL^{\ell_p} h_{  \ell_0 \cdots \ell_p }(y) = \frac{1}{\lambda( C(\ell_0, \ldots, \ell_p) )} \frac{1}{ (
	T^{ \ell_p } )' y_{  \ell_0 \cdots \ell_p } },
$$
where $y_{  \ell_0 \cdots \ell_p }$ is the unique preimage of $y \in Y$ in 
$C(\ell_0, \ldots, \ell_p)$ under $T^{\ell_p}$. Let $y, z \in Y$.
Using the chain rule, we see that 
\begin{align*}
	&|  \log \cL^{\ell_p} h_{  \ell_0 \cdots \ell_p }(y) - \log \cL^{\ell_p} h_{  \ell_0 \cdots \ell_p }(z) | \\
	&\le | \log (T^{\ell_0})' y_{  \ell_0 \cdots \ell_p } - (T^{\ell_0})' z_{  \ell_0 \cdots \ell_p }  | \\
	&+ \sum_{j=1}^{p} | \log  (T^{ \ell_j - \ell_{j-1} }  )'( T^{\ell_{j-1}}  y_{  \ell_0 \cdots \ell_p }  )  - \log 
	(T^{ \ell_j - \ell_{j-1} }  )'( T^{\ell_{j-1}}  z_{  \ell_0 \cdots \ell_p }  )   |,
\end{align*}
where the sum vanishes if $p = 0$. For $x \in \{  y_{  \ell_0 \cdots \ell_p } ,  z_{  \ell_0 \cdots \ell_p }   \}$, we have that 
$T^{\ell_{j-1}}  x  \in \delta_{ \ell_{j} - \ell_{j-1} }$ if $1 \le j \le p$, 
and $x \in \Delta_{\ell_0} \cup \delta_{\ell_0}$. Therefore, using \eqref{eq:dist} and \eqref{eq:unif_exp}, 
we obtain 
\begin{align*}
	&|  \log \cL^{\ell_p} h_{  \ell_0 \cdots \ell_p }(y) - \log \cL^{\ell_p} h_{  \ell_0 \cdots \ell_p }(z) | \\
	&\le  C_\gamma \sum_{j=0}^p |  T^{\ell_j} ( y_{  \ell_0 \cdots \ell_p }  ) -  T^{\ell_j} ( z_{  \ell_0 \cdots \ell_p }  )   | \le C_\gamma \sum_{j = 0}^p  \Lambda^{  - p + j } |y - z | 
	\le C_\gamma |y - z|.
\end{align*}
We conclude that for sufficiently large $K_2 > 0$ depending only on $\gamma$, the 
measure 
$$
\mu:= ( T^{ \ell_p} )_* ( \lambda |_{  C(\ell_0, \ldots, \ell_p) } ) / C(\ell_0, \ldots, \ell_p)
$$ 
satisfies $| \mu |_{\LL} \le K_2$. In particular, since $\mu$ is a probability measure, 
$d \mu / d m$ is bounded, so that 
\eqref{eq:estim_partition} implies the tail bound 
$$\mu( \tau \ge n ) \le C_\gamma m( \tau \ge n) \le C_\gamma n^{ - 1 / (\gamma - 1)}.$$
Moreover, from \eqref{eq:K1} we see that $\mu$ satisfies \eqref{eq:regular_measure}.
Now, by
\begin{align}\label{eq:ml_appli}
\begin{split}
&\Vert    \cL^{ i_{m+1} - \ell_p } \cL^{  \ell_p } ( h_{  \ell_0 \cdots \ell_p } ) - 1  \Vert_{L^1( \lambda )} 
\le C_\gamma (   i_{m+1} - \ell_p )^{  - 1 / ( \gamma - 1)  } \\
&\le C_\gamma (   i_{m+1} - \ell_\# )^{  - 1 / ( \gamma - 1)  } \le C_\gamma (   i_{m+1} - i_m )^{  
	- 1 / ( \gamma - 1) }.
\end{split}
\end{align}
We arrive at the upper bound
\begin{align}\label{eq:estim_I}
|\cI| \le  \Vert F \Vert_{ \Lip } \lambda(  C( \ell_0, \ldots, \ell_p)  )   
C_\gamma (   i_{m+1} - i_m )^{  - 1 / ( \gamma - 1) }.
\end{align}

\noindent\textbf{Conclusion:} Assembling \eqref{eq:step_2-2} and \eqref{eq:estim_I}, 
\begin{align*}
	&\biggl| 
	\int_{  M  } H (x,x) \, d \lambda(x)  - \iint_{  M^2  } H (x,y) \, d \lambda(x) \, d \lambda(y)
	\biggr| \\
	&\le   \sum_{p=0}^{i_*} \sum_{\ell_0 < \ldots < \ell_p, \, \ell_p \le \ell_{\#}}  \Vert F \Vert_{ \Lip } \lambda(  C( \ell_0, \ldots, \ell_p)  )   C_\gamma (   i_{m+1} - i_m )^{  - 1 / ( \gamma - 1) }.  \\
	&+ C_\gamma
	 \Vert F  \Vert_{ \Lip } (  \Vert f \Vert_{ \Lip }  + 1) (  i_{m+1} - i_m )^{ - 1/( \gamma - 1)  } \\
	 &\le C_\gamma   \Vert F  \Vert_{ \Lip } (  \Vert f \Vert_{ \Lip }  + 1) (  i_{m+1} - i_m )^{ - 1/( \gamma - 1)  }.
\end{align*}
This completes the proof of \eqref{eq:toshow_fcb}.

\subsection*{Data availability} No datasets were generated or analyzed in this study.

\appendix 

\section{Proof of Proposition \ref{prop:estim_I}}\label{sec:proof_estimates_Ei}

\noindent ($E_1$): Let $|i - n| = m$, $0 \le i < N$, $0 \le n,m < N$, $s \in [0,1]$.  By \eqref{eq:Dk_fml},
\begin{align*}
	&\mu ( X_n 
	X_i   \xi_{i,n,m}(s)  )  \\ 
	&= \mu \biggl\{  
	X_n X_i ( 
	f''(   W_{n,m} + s Y_{n,m}     )[   \theta_i  ,  \theta_n ] \,   - f''(   W_{n,m}    )[   \theta_i ,  \theta_n]
	)
	\biggr\} \\
	&= -    \int_0^\infty  e^{-2u}   \mu \biggl\{  X_n X_i \bfE \biggl(  g''(  (W_{n,m} + s Y_{n,m}) e^{-u} + \sigma(u) Z )[\theta_i, \theta_n] \\
	&-  g''(  W_{n,m} e^{-u} + \sigma(u) Z )[\theta_i, \theta_n] 
	\biggr)  \biggr\} \, d u \\
	&= -   \int_0^\infty  e^{-2u}    \mu \biggl\{  X_n X_i   H_g(  S_{n,m}, T_{n,m} ; e^{-u},se^{-u}, \sigma(u),n)  \biggr\}   \, d u
\end{align*}
where $H_g$ is defined as in \eqref{eq:Hg}. Hence, by (A2), 
\begin{align*}
	|E_1| &= \biggl| \sum_{n=0}^{N-1} \sum_{m=1}^{N-1}  \sum_{ \substack{|i - n| = m \\ 
			0 \le i < N } }  \int_0^1  \mu ( \tilde{X}_n 
	\tilde{X}_i   \xi_{i,n,m}(s)  )  \, ds  \biggr| \\ 
	&\le \sum_{n=0}^{N-1} \sum_{m=1}^{N-1}  \sum_{ \substack{|i - n| = m \\ 
			0 \le i < N } }  \int_0^1  
	B^{-1}  \int_0^\infty  e^{-2u}   \biggl| \mu \biggl\{  X_n X_i   H_g(  S_{n,m}, T_{n,m} ; e^{-u},se^{-u}, \sigma(u),n)  \biggr\} \biggr|   \, d u
	\, ds \\
	&\lesssim C_2 B^{-3/2} N \sum_{m=1}^{N-1} \rho(m).
\end{align*}

\noindent  ($E_2$): Recalling \eqref{eq:phi_bound}, we have that 
\begin{align*}
	|E_2| &= \biggl|  \sum_{n=0}^{N-1}      \int_0^1   \mu( |  \tilde{X}_n^2  \xi_{n, n,0}(s) |)   \, ds  \biggr| \\
	&\le B^{-1} \sum_{n=0}^{N-1}      \int_0^1  \mu(   | X_n^2   (   f''(   W_{n,0} + s Y_{n}    )[   \theta_n  ,  \theta_n ] \,   - f''(   W_{n,0}    )[   \theta_n  ,  \theta_n ] )  | )  \, ds  \\
	&\lesssim B^{-3/2} N  L^3 \Vert g \Vert_{\sM}.
\end{align*}

\noindent ($E_3$): As in the case of $E_1$, using \eqref{eq:Dk_fml} we see that 
\begin{align}\label{eq:repr_mu}
\mu ( X_n  X_i    \overline{\xi_{i, n,k}}   ) = - \int_0^\infty  e^{-2u}    \mu \biggl\{  X_n X_i   \overline{H_g(  S_{n,k}, T_{n,k} ; e^{-u},e^{-u}, \sigma(u),i, n)}  \biggr\}   \, d u.
\end{align}
Hence,
\begin{align*}
	&|E_3|  \\
	&=  \biggl| \sum_{n=0}^{N-1}  \sum_{m=1}^{N-1}   \sum_{k=m + 1}^{2m}    \sum_{ \substack{|i - n| = m \\ 
			0 \le i < N } } \mu ( 
	\tilde{X}_n   \tilde{X}_i    \overline{\xi_{i, n,k}}   ) \biggr|  \\
	&\le B^{-1} \sum_{n=0}^{N-1}  \sum_{m=1}^{N-1}   \sum_{k=m + 1}^{2m}   
	\sum_{ \substack{|i - n| = m \\ 
			0 \le i < N } }
	\int_0^\infty  e^{-2u}   \biggl|  \mu \biggl\{  X_n X_i   H_g(  S_{n,k}, T_{n,k} ; e^{-u},e^{-u}, \sigma(u),i, n)  \biggr\}  \biggr|  \, d u  \\
	&+ B^{-1} \sum_{n=0}^{N-1}  \sum_{m=1}^{N-1}   \sum_{k=m + 1}^{2m}   
	\sum_{ \substack{|i - n| = m \\ 
			0 \le i < N } }
	\int_0^\infty  e^{-2u}   \biggl|  
	\mu \biggl\{  H_g(  S_{n,k}, T_{n,k} ; e^{-u},e^{-u}, \sigma(u),i, n) \biggr\} \mu \biggl\{  X_n X_i   \biggr\}   \biggr|  \, d u  \\
	&= E_{3,1} + E_{3,2}.
\end{align*}
By (A2),
$$
E_{3,1}  \lesssim C_2 B^{-3/2} N  \Vert  g \Vert_{\sM} \sum_{m=1}^{N-1} m \rho(m).
$$
Further, 
\begin{align*}
E_{3,2} &\lesssim 
B^{-3/2}  L  \Vert g \Vert_\sM   \sum_{n=0}^{N-1}  \sum_{m=1}^{N-1}  m   
\sum_{ \substack{|i - n| = m \\ 
		0 \le i < N } }  |  \mu \{  X_n X_i   \}   | 
	\lesssim C_1 B^{-3/2} N   L  \Vert g \Vert_\sM  \sum_{m=1}^{N-1} m \rho(m),
\end{align*}
where \eqref{eq:bounds_Hg} and (A1) were used to obtained the first 
and second inequality, respectively.
Consequently,
$$
|E_3| \lesssim (C_1 L  + C_2) B^{-3/2} N  \Vert g \Vert_\sM  \sum_{m=1}^{N-1} m \rho(m).
$$

\noindent ($E_4$): Combining \eqref{eq:repr_mu} with (A3) yields 
\begin{align*}
	|E_4| &= \biggl| \sum_{n=0}^{N-1}  \sum_{m=1}^{N-1}   \sum_{k= 2m + 1}^{N - 1}    \sum_{ \substack{|i - n| = m \\ 
			0 \le i < N } } \mu (
	\tilde{X}_{n}   \tilde{X}_i     \overline{\xi_{i, n,k}} ) \biggr| \\
	&\le B^{-1} \sum_{n=0}^{N-1}  \sum_{m=1}^{N-1}   \sum_{k= 2m + 1}^{N - 1}    \sum_{ \substack{|i - n| = m \\ 
			0 \le i < N } } 
	\int_0^\infty  e^{-2u}   \biggl|   \mu \biggl\{  X_n X_i   \overline{H_g(  S_{n,k}, T_{n,k} ; e^{-u}, e^{-u}, \sigma(u),i, n)}  \biggr\}   \biggr| \, d u \\
	&\lesssim B^{-3/2} N  C_3 \Vert g \Vert_{\sM} \sum_{m=1}^{N-1} \sum_{k=2m + 1}^{N-1}  \rho(k -m)
	\lesssim B^{-3/2} N  C_3 \Vert g \Vert_{\sM} \sum_{m=1}^{N-1} m \rho (m).
\end{align*}

\noindent ($E_5$),  ($E_6$), ($E_7$): Similarly, we derive the following estimates for the remaining 
three terms. Their verification is left to the reader.
\begin{align*}
	|E_5| &=  \biggl| \sum_{n=0}^{N-1}    \sum_{k= 1}^{N - 1}   \mu (
	\tilde{X}_n^2     \overline{\xi_{n, n,k}}  ) \biggr| \lesssim B^{-3/2} N C_3 \Vert g \Vert_\sM 
	\sum_{k=1}^{N-1} \rho(k),\\
	|E_6| &= \biggl| 
	\sum_{n=0}^{N-1}  \sum_{m=1}^{N-1} \sum_{k=0}^m  \sum_{ \substack{|i - n| = m \\ 
			0 \le i < N } } \mu( \tilde{X}_n
	\tilde{X}_m  )   \mu(  \xi_{i,n,k}  )
	\biggr| \lesssim C_1 B^{-3/2} N L \Vert g \Vert_\sM   \sum_{m=1}^{N-1} m \rho(m), \\ 
	|E_7| &= \biggl| \sum_{n=0}^{N-1}  \mu( \tilde{X}_n^2 ) \mu(  \xi_{n,n,0}  ) \biggr|  
	\lesssim   B^{-3/2} N  L^3 \Vert g \Vert_{\sM}.
\end{align*}

The desired estimate in Proposition \ref{prop:estim_I} follows by combining the above 
estimates on $E_i$, $1 \le i \le 7$.

\bigskip
\bigskip
\bibliography{barbour}{}
\bibliographystyle{plainurl}

\vspace*{\fill}

\end{document}